\documentclass[a4paper,10pt]{amsart}
\usepackage[utf8]{inputenc}
\usepackage{amsfonts, amsmath, amssymb, amsthm, mathtools}
\usepackage{xcolor}
\definecolor{darkblue}{rgb}{0,0,0.6}
\usepackage{tikz}
\usetikzlibrary{matrix,arrows,trees}
\usepackage{tikz-cd}
\usepackage[cmtip]{xypic}
\usepackage[ocgcolorlinks,colorlinks=true, citecolor=darkblue, filecolor=darkblue, linkcolor=darkblue, urlcolor=darkblue]{hyperref}
\usepackage[nameinlink, capitalise]{cleveref}

\hypersetup{bookmarksopen=true}

\numberwithin{equation}{section}

\swapnumbers
\theoremstyle{definition}
\newtheorem{ddd-alt}[equation]{Definition}
\newtheorem{ex-alt}[equation]{Example}
\newtheorem{const-alt}[equation]{Construction}
\newtheorem{rem-alt}[equation]{Remark}
\theoremstyle{plain}
\newtheorem{theorem}[equation]{Theorem}
\newtheorem{prop}[equation]{Proposition}
\newtheorem{lem}[equation]{Lemma}
\newtheorem{kor}[equation]{Corollary}

\crefname{theorem}{Theorem}{Theorems}

\newenvironment{ddd}    
{%
	\pushQED{\qed}\begin{ddd-alt}}
	{\popQED\end{ddd-alt}}
	
\newenvironment{ex}    
{%
	\pushQED{\qed}\begin{ex-alt}}
	{\popQED\end{ex-alt}}
	
\newenvironment{rem}    
{%
	\pushQED{\qed}\begin{rem-alt}}
	{\popQED\end{rem-alt}}

\newenvironment{const}    
{%
	\pushQED{\qed}\begin{const-alt}}
	{\popQED\end{const-alt}}

\renewcommand{\epsilon}{\varepsilon}
\renewcommand{\theta}{\vartheta}
\renewcommand{\phi}{\varphi}

\renewcommand{\emptyset}{\varnothing}

\newcommand{\bbN}{\mathbb{N}}

\newcommand{\bA}{\mathbf{A}}

\newcommand{\bC}{\mathbf{C}}
\newcommand{\bD}{\mathbf{D}}

\newcommand{\bK}{\mathbf{K}}

\newcommand{\bN}{\mathbf{N}}

\newcommand{\bR}{\mathbf{R}}

\newcommand{\bW}{\mathbf{W}}

\newcommand{\cB}{\mathcal{B}}
\newcommand{\cC}{\mathcal{C}}
\newcommand{\ccD}{\mathcal{D}}
\newcommand{\cE}{\mathcal{E}}

\newcommand{\ccH}{\mathcal{H}}

\newcommand{\cM}{\mathcal{M}}

\newcommand{\cU}{\mathcal{U}}
\newcommand{\cV}{\mathcal{V}}

\newcommand{\cX}{\mathcal{X}}
\newcommand{\cY}{\mathcal{Y}}

\newcommand{\fN}{\mathfrak{N}}

\newcommand{\abs}[1]{\lvert #1 \rvert}

\DeclareMathOperator*{\colim}{colim}
\DeclareMathOperator{\cofib}{cofib}

\DeclareMathOperator{\id}{id}

\DeclareMathOperator{\Fun}{Fun}

\DeclareMathOperator{\ind}{Ind}
\DeclareMathOperator{\Map}{Map}

\DeclareMathOperator{\res}{Res}

\DeclareMathOperator{\ass}{\alpha}
\DeclareMathOperator{\AX}{\mathbf{A}\cX}
\DeclareMathOperator{\cells}{\diamond}
\DeclareMathOperator{\cof}{co}
\DeclareMathOperator{\diag}{diag}

\DeclareMathOperator{\hcnerve}{N^{hc}}

\DeclareMathOperator{\Idem}{Idem}
\DeclareMathOperator{\inc}{inc}
\DeclareMathOperator{\Ind}{Ind}

\DeclareMathOperator{\orb}{Or}

\DeclareMathOperator{\sk}{sk}
\DeclareMathOperator{\stab}{Stab}
\DeclareMathOperator{\UAX}{\mathrm{U}\mathbf{A}\cX}
\DeclareMathOperator{\uloc}{\cU_\mathrm{loc}}
\DeclareMathOperator{\wK}{\bK^{\mathrm{W}}}
\DeclareMathOperator{\wuloc}{\cU^{\mathrm{W}}_\mathrm{loc}}

\newcommand{\sSet}{\mathrm{sSet}}
\newcommand{\BC}{\mathbf{BornCoarse}}
\newcommand{\catex}{\mathrm{Cat}^{\mathrm{ex}}_\infty}
\newcommand{\catinf}{\mathrm{Cat}_\infty}
\newcommand{\catperf}{\mathrm{Cat}^{\mathrm{perf}}_\infty}
\newcommand{\catrperf}{\mathrm{Cat}^{\mathrm{Rperf}}_\infty}
\newcommand{\catrex}{\mathrm{Cat}^{\mathrm{Rex}}_\infty}

\newcommand{\cf}{\mathrm{cf}}
\newcommand{\ccw}[3]{\bC^{#1}(#2;#3)}
\newcommand{\CW}{\mathbf{CW}}
\newcommand{\f}{\mathrm{f}}
\newcommand{\free}{\mathrm{free}}
\newcommand{\Fin}{\mathbf{Fin}}

\newcommand{\gen}[1]{\langle #1 \rangle}
\newcommand{\lf}{\mathrm{lf}}
\newcommand{\mloc}{\cM_{\mathrm{loc}}}
\newcommand{\RelCat}{\mathrm{RelCat}_{(2,1)}}
\newcommand{\RelCatone}{\mathrm{RelCat}}
\newcommand{\ret}[4]{\bR^{#1}_{#4}(#2;#3)}
\newcommand{\retuc}[3]{\bR^{#1}_{#3}(#2)}
\newcommand{\Sp}{\mathrm{Sp}}
\newcommand{\Top}{\mathrm{Top}}
\newcommand{\tr}{\mathrm{tr}}
\newcommand{\Vcyc}{\mathbf{VCyc}}
\newcommand{\Waldone}{\mathrm{Wald^{\mathrm{ho}}}}
\newcommand{\Wald}{\mathrm{Wald^{\mathrm{ho}}_{(2,1)}}}

 \newcommand{\FDC}{\mathbf{FDC}}

\newcounter{commentcounter}

\crefname{section}{Section}{Sections}

\title{Split injectivity of A-theoretic assembly maps}

\author{Ulrich Bunke}
\address{Fakult{\"a}t f{\"u}r Mathematik, Universit{\"a}t Regensburg, 93040 Regensburg, Germany}
\email{ulrich.bunke@mathematik.uni-regensburg.de} 

\author{Daniel Kasprowski}
\address{Rheinische Friedrich-Wilhelms-Universit\"at Bonn, Mathematisches Institut,\newline\indent Endenicher Allee 60, 53115 Bonn, Germany}
\email{kasprowski@uni-bonn.de}

\author{Christoph Winges}
\address{Rheinische Friedrich-Wilhelms-Universit\"at Bonn, Mathematisches Institut,\newline\indent Endenicher Allee 60, 53115 Bonn, Germany}
\email{winges@math.uni-bonn.de}

\date{}

\begin{document}
 \begin{abstract}
 We construct an equivariant coarse homology theory arising from the algebraic $K$-theory of spherical group rings
 and use this theory to derive split injectivity results for associated assembly maps.
 
 On the way, we prove that the fundamental structural theorems for Waldhausen's algebraic $K$-theory functor carry over to its nonconnective counterpart defined by Blumberg--Gepner--Tabuada.
\end{abstract}

\maketitle

\section{Introduction}
For a group $G$, let $P$ be the total space of a principal $G$-bundle and let $\bA$ denote the functor of nonconnective $A$-theory (taking values in the $\infty$-category of spectra). Then
$P$ gives rise to an $\orb(G)$-spectrum $\bA_P$ sending a transitive $G$-set $S$ to the spectrum $\bA(P \times_G S)$.  We will show the following split injectivity results for $\bA_P$. 

\begin{theorem}
	\label{thm:injectivity-linear}
Let $G$ be a group and assume
\begin{enumerate}
\item $G$ is finitely generated;
\item $G$ admits a finite-dimensional model for $\underline{E}G$;
\item one of:
\begin{enumerate}
\item $G$ is a subgroup  of a linear group over a commutative ring
\item $G$ is a subgroup of a virtually connected Lie group.
\end{enumerate}\end{enumerate}
Then the $A$-theoretic assembly map for the family of finite subgroups
	\[ \ass_{\bA_{P}}^\Fin \colon \colim_{ {\orb_\Fin(G)}} \bA_P \to \bA_{P}(*) \simeq \bA(P/G) \]
	is split injective.
\end{theorem}

\begin{theorem}\label{thm:split-injective-relative}
	The relative $A$-theoretic assembly map from the family of finite subgroups to the family of virtually cyclic subgroups
	\[ \ass_{\bA_P}^{\Fin,\Vcyc} \colon \colim_{\orb_\Fin(G)} \bA_P \to \colim_{\orb_\Vcyc(G)} \bA_P \]
	is split injective.
\end{theorem}

Similar results about the $K$- and $L$-theoretic assembly maps for discrete group rings were originally obtained by Carlsson--Pedersen \cite{CP95} and subsequently generalized by Bartels and Rosenthal \cite{BR07}, Ramras, Tessera and Yu \cite{RTY14} and Kasprowski \cite{Kasprowski15}. The analog of \cref{thm:split-injective-relative} in the case of discrete group rings is originally due to Bartels \cite{Bartels03}.   

More precisely, we will show in \cref{thm:a-cp} that $\bA_P$ is a (hereditary) CP-functor, a notion introduced in \cite{injectivity}. We then apply results from \cite{injectivity} in order to deduce the above theorems in \cref{rgioergegegerge}.

Recall from \cite[Definition~1.5]{injectivity} that the family of subgroups $\FDC$ consists of those subgroups $H$ of $G$ such that the family $\{F\backslash H\mid F\leq H \text{ finite}\}$ has finite decomposition complexity as defined by Guentner, Tessera and Yu in \cite{GTY12, GTY13}.

\begin{theorem}
	\label{thm:FDC}
Assume that $G$ admits a finite-dimensional model $\underline EG$.
The relative $A$-theoretic assembly map from the family of finite subgroups to the family $\FDC$
\[ \ass_{\bA_P}^{\Fin,\FDC} \colon \colim_{\orb_\Fin(G)} \bA_P \to \colim_{\orb_\FDC(G)} \bA_P \]
is split injective.
\end{theorem}
For more details on this and a slightly more general result see \cite[Theorem~1.11]{injectivity}. We also obtain the following $A$-theoretic analog of \cite[Theorem~1.15]{injectivity}.
\begin{theorem}
	\label{thm:relhyp}  {We assume the following:}
	\begin{enumerate}
		\item $G$ admits a finite-dimensional model for $\underline{E}G$; 
		\item $G$  is relatively hyperbolic  {to groups $H_1,\ldots,H_n$;}
		\item for every $i\in \{1,\dots n\}$ we  have one of:
		\begin{enumerate} 
			\item $H_i$ is contained in $\FDC$;
			\item  for every total space $P_i$ of a principal $H_i$-bundle the $A$-theoretic assembly map for the family of virtually cyclic subgroups
			\[ \ass_{\bA_{P_i}}^\Vcyc \colon \colim_{\orb_\Vcyc(H_i)} \bA_{P_i} \to \bA_{P_i}(*) \simeq \bA(P_i/H_i) \]
			is an equivalence. \end{enumerate}\end{enumerate}
	Then the $A$-theoretic assembly map for the family of finite subgroups of $G$
	\[ \ass_{\bA_P}^\Fin \colon \colim_{\orb_\Fin(G)} \bA_P \to \bA_{P}(*) \simeq \bA(P/G) \]
	is split injective.
\end{theorem}

Finally, our methods imply the following result which was previously obtained by Barwick \cite[Example~C]{Barwick17} and Malkiewich and Merling \cite[Section~4]{MM}.
\begin{theorem}\label{thm:mackey-functor}
	If $G$ is a finite group, then the $\orb(G)$-spectrum $\bA_P$ extends to a spectral Mackey functor.
\end{theorem}

To prove \cref{thm:a-cp}, we recast the definitions of coarse versions of $A$-theory given by Weiss \cite{Weiss02} and Ullmann and Winges \cite{UW} in the setting of bornological coarse spaces.
Since this construction relies on a sufficiently well-behaved version of nonconnective algebraic $K$-theory,
\cref{sec:k-theory} discusses the properties of the nonconnective $K$-theory functor introduced by Blumberg--Gepner--Tabuada \cite{BGT13} as a functor on Waldhausen categories.
Specifically, we show that the validity of the Additvity, Fibration, Approximation and Cofinality theorems are preserved in passing to the nonconnective version.

The necessary translation of the categories of controlled retractive spaces from \cite{UW} to the setting of bornological coarse spaces is done in \cref{sec:controlled-retractive-spaces}.
Using the results from \cref{sec:k-theory}, we are then able to give a streamlined proof of the fact that coarse $A$-homology is a coarse homology theory in \cref{sec:a-homology}.

The final \cref{sec:injectivity} establishes the last properties needed to obtain \cref{thm:a-cp}. In particular, we discuss the construction of transfer maps.

\tableofcontents

\subsection*{Acknowledgments} U.~Bunke  was supported by the SFB 1085 ``Higher Invariants'' funded by the Deutsche Forschungsgemeinschaft DFG.
C.~Winges was supported by the Max Planck Society and Wolfgang L\"uck's ERC Advanced Grant ``KL2MG-interactions" (no.~662400).
D.~Kasprowski and C.~Winges are members of the Hausdorff Center for Mathematics at the University of Bonn.


\section{Algebraic \texorpdfstring{$K$}{K}-theory}\label{sec:k-theory}

The algebraic $K$-theory functor originally defined by Waldhausen \cite{Wald85} takes values in connective spectra.
For the applications in \cref{sec:a-homology} and \cref{sec:injectivity}, we require a nonconnective version of algebraic $K$-theory.
A nonconnective version of $A$-theory can be derived from Waldhausen's connective $K$-theory functor using the methods of \cref{sec:controlled-retractive-spaces,sec:a-homology}, cf.~\cite[Section~5]{UW}.
However, we prefer to base our discussion on the axiomatic approach of Blumberg--Gepner--Tabuada~\cite{BGT13}.

Waldhausen's $K$-theory functor is particularly useful to prove structural results. The key tools in Waldhausen's approach
are the Additivity theorem~\cite[Theorem~1.4.2]{Wald85}, the Fibration theorem~\cite[Theorem~1.6.4]{Wald85}, the Approximation theorem~\cite[Theorem~1.6.7]{Wald85}
and the Cofinality theorem~\cite[1.10.1]{TT90}, \cite[Theorem~1.6]{Vogell90}.
In the present section we provide analogs of these theorems for the  {universal} nonconnective $K$-theory of Blumberg--Gepner--Tabuada:
the Additivity theorem (\cref{thm:additivity}) holds as a corollary of the Fibration theorem (\cref{thm:fibration}); the appropriate analog of the Approximation theorem is recalled in \cref{thm:approximation}, and a version of the Cofinality theorem is given in \cref{thm:cofinality}.

We also discuss the compatibility of $K$-theory with infinite products---the connective case was originally established by Carlsson \cite{Carlsson95}---in \cref{sec:products}.

\subsection{Waldhausen categories as right-exact \texorpdfstring{$\infty$}{infinity}-categories}\label{sec:right-exact}

In order to be able to employ
the theory developed in \cite{BGT13}, we consider a class of Waldhausen categories (called homotopical) whose homotopy theory can be adequately described in terms of $\infty$-categories.

\begin{ddd}\label{def:homotopical-Waldhausen-cat}
 Let $\bC$ be a Waldhausen category.
 \begin{enumerate}
  \item $\bC$ \emph{admits factorizations} if every morphism in $\bC$ can be factored into a cofibration followed by a weak equivalence; we assume no functoriality.
  \item $\bC$ is \emph{homotopical} if it admits factorizations and the weak equivalences satisfy the two-out-of-six property. \qedhere
 \end{enumerate}
\end{ddd}

Recall that the two-out-of-six property means that if
\[  {C_0 \xrightarrow{x_1} C_1 \xrightarrow{x_2} C_2 \xrightarrow{x_3} C_3} \]
are composable morphisms such that both $x_2\circ x_1$ and $x_3\circ x_2$ are weak equivalences,
then also $x_1$, $x_2$ and $x_3$ (and hence also $x_3\circ x_2\circ x_1$) are weak equivalences.

\begin{rem}
 Homotopical Waldhausen categories as defined above are precisely the Waldhausen categories considered in \cite{BGT13}.
 The term ``homotopical" has been borrowed from \cite[Chapter~5]{DHKS04}.
 
 Let us comment shortly why it is sensible to restrict to this class of Waldhausen categories.
 
The existence of factorizations in $\bC$ is the most natural condition to guarantee that the $\infty$-categorical localization $\bC[w\bC^{-1}]$ admits all finite colimits,
 which makes it amenable to the methods of \cite{BGT13};  see also \cite[Definition~9.30]{BGT13}.
 
 Moreover, note that equivalences in any $\infty$-category satisfy the two-out-of-six property. In order to ensure that notions defined in terms of Waldhausen categories match up with their $\infty$-categorical counterparts, it is natural to require that the localization $\bC\to \bC[w\bC^{-1}]$ detects weak equivalences.
 By \cite[Corollary~7.5.19]{CisBook}, this is the case if and only if the weak equivalences in $\bC$ satisfy the two-out-of-six property, see also \cref{prop:exact-localization}\eqref{it:exact-localization-1}.
 See also \cite[Lemma~A.2.3, Theorem~B.5.1 and Theorem~6.4]{BM11} and \cite{Weiss99} for similar observations with respect to the hammock localization.
\end{rem}

 Let $\Waldone$ denote the category of homotopical Waldhausen categories and exact functors.
As an auxiliary tool, we introduce the category $\RelCatone$ of relative categories and functors between relative categories.
There is a forgetful functor
\[ u \colon \Waldone \to \RelCatone,\quad (\bC,\cof\bC,w\bC) \mapsto (\bC,w\bC). \]
Composing with the functor $\RelCatone \to \catinf$ which sends a relative category to its localization,
$u$ induces a functor \begin{equation}\label{vewrverwefwwefwfew}
\ell' \colon \Waldone \to \catinf\ .
\end{equation}

\begin{rem}\label{rem:explicit-localization}
In this remark we provide a point-set model for the functor \eqref{vewrverwefwwefwfew}. We model $\infty$-categories by quasi-categories and write $\hcnerve(\cC)$ for the $\infty$-category represented by the homotopy-coherent nerve of  a fibrant simplicial category $\cC$. In order to apply this to an  ordinary category $\cC$ we consider it  as a simplicial category with discrete mapping spaces.
By abuse of notation, we also use $\cC$ to denote $\hcnerve(\cC)$ for an ordinary category $\cC$ in the main body of the paper.

Let $\sSet^+$ denote the category of marked simplicial sets \cite[Definition~3.1.0.1]{LurieHTT} equipped with the \emph{marked model structure} from \cite[Propositon~3.1.3.7]{LurieHTT}.
This is a combinatorial simplicial model structure by \cite[Proposition~3.1.3.7 and Corollary~3.1.4.4]{LurieHTT}
in which every object is cofibrant and whose fibrant objects are precisely those simplicial sets which are quasicategories equipped with their subcategory of equivalences \cite[Proposition~3.1.4.1]{LurieHTT}.
Following \cite[Chapter~3]{LurieHTT}, we define the $\infty$-category of $\infty$-categories
\[ \catinf := \hcnerve((\sSet^+)^\cf) \]
as the homotopy coherent nerve of the subcategory of cofibrant-fibrant objects in $\sSet^+$.

Let $\cC$ be an $\infty$-category and let $w\cC$ be a wide subcategory.
Recall that a localization of $\cC$ at $w\cC$ is a functor $l \colon \cC \to \cC[w\cC^{-1}]$ such that restriction along $l$ defines for every $\infty$-category $\ccD$ an equivalence of $\infty$-categories
\[ l^* \colon \Fun(\cC[w\cC^{-1}],\ccD) \xrightarrow{\sim} \Fun_{w\cC}(\cC,\ccD), \]
where $\Fun_{w\cC}$ denotes the full subcategory of functors sending all morphisms in $w\cC$ to equivalences.
Localizations always exist and are essentially unique \cite[Proposition~7.1.3]{CisBook}.
Note that $l^*$ restricts to an equivalence on maximal Kan complexes, and thus induces an equivalence
\[ l^* \colon \Map_{\catinf}(\cC_1,\ccD) \xrightarrow{\sim} \Map_{\catinf}^{w\cC}(\cC,\ccD), \]
where $\Map_{\catinf}^{w\cC}$ denotes the collection of all components containing functors which send all morphisms in $w\cC$ to equivalences.
It follows formally that the localization $l$ is equivalently described by the latter universal property.
Consequently, any fibrant replacement of $(\cC,w\cC)$ as a marked simplicial set models the $\infty$-categorical localization $l \colon \cC \to \cC[w\cC^{-1}]$, cf.~\cite[Remark~1.3.4.2]{LurieHA}.

Composing $u$ with the functor
\[ L \colon \RelCatone \to \sSet^+, \quad (\bC,w\bC) \mapsto (\hcnerve(\bC),\hcnerve(w\bC)) \]
and any fibrant replacement functor $R \colon \sSet^+ \to \sSet^+$ produces a concrete model
\begin{equation}\label{eblkjrlkfjrlkf43r3r34r43r43r}
 \ell' := \hcnerve(R \circ L \circ u) \colon \hcnerve(\Waldone) \to \catinf
\end{equation}
of the functor \eqref{vewrverwefwwefwfew}.
Note that
\[ \ell'(\bC,\cof\bC,w\bC) \simeq \bC[w\bC^{-1}] \]
by the previous discussion.
\end{rem}

Recall that an $\infty$-category is \emph{right-exact} if it has a zero object and admits all finite colimits.
A functor between right-exact $\infty$-categories is \emph{exact} if it preserves all finite colimits.
Denote by $\catrex \subseteq \catinf$ the subcategory of right-exact $\infty$-categories and exact functors.
By the dual of \cite[Proposition~7.5.6]{CisBook}, $\ell'$ in \eqref{vewrverwefwwefwfew} factors via the inclusion $\catrex \subseteq \catinf$ to give a localization functor \begin{equation}\label{vrtojp4gf4rvvfeveve}
 \ell \colon \Waldone \to \catrex. 
\end{equation}

Let $\bC$ be a homotopical Waldhausen category and let $\ccD$ be a right-exact $\infty$-category.
\begin{ddd}[{See \cite[Definition~7.5.2 and Example~7.5.9]{CisBook}}]\label{def:exact-functor}
A functor $\bC \to \ccD$ is \emph{exact}
 if it preserves zero objects, sends pushouts along cofibrations in $\bC$ to pushouts in $\bD$, and
 maps all morphisms in $w\bC$ to equivalences in $\ccD$.
\end{ddd}

\begin{prop}\label{prop:exact-localization}
 Let $\bC$ be a homotopical Waldhausen category.
 \begin{enumerate}
  \item\label{it:exact-localization-1} The localization functor $l \colon \bC \to \ell(\bC)$ is exact and detects weak equivalences.
  \item\label{it:exact-localization-2} For every right-exact $\infty$--category $\ccD$, restriction along $l$ induces an equivalence
   \[ l^* \colon \Fun^{\mathrm{ex}}(\ell(\bC), \ccD) \xrightarrow{\sim} \Fun^{\mathrm{ex}}(\bC, \ccD) \]
   between the full subcategories of exact functors.
 \end{enumerate}
\end{prop}
\begin{proof}
 The localization functor $l$ is exact by the dual of \cite[Proposition~7.5.6]{CisBook}.
 If $l(f)$ is an equivalence for some morphism $f \colon X \to Y$ in $\bC$, $l(f)$ defines an isomorphism in the homotopy category of $\ell(\bC)$.
 Applying \cite[Corollary~7.5.19]{CisBook} twice, it follows that there exist morphisms $g \colon Y \to X$ and $f' \colon X \to Y$ such that both $fg$ and $gf'$ are weak equivalences in $\bC$.
 By the two-out-of-six property, it follows that $f$ is a weak equivalence.
 This proves part \eqref{it:exact-localization-1}.
 
 Part \eqref{it:exact-localization-2} is precisely the (dualized) assertion of \cite[Proposition~7.5.11]{CisBook}.
\end{proof}

\subsection{The fundamental theorems}\label{sec:fundamental-theorems}

We can now extend the universal localizing invariant
\[  {\uloc \colon \catex \to \mloc} \]
of \cite{BGT13} to the category of homotopical Waldhausen categories as indicated in \cite[Definition~9.30]{BGT13}.
Recall the stabilization functor  
\[ \stab \colon \catrex \to \catex, \]
which is the left adjoint of the fully faithful inclusion functor $\catex \to \catrex$, see \cite[Section~9.3]{BGT13}.
The stable $\infty$-category $\stab(\cC)$ admits the description
\begin{equation}\label{eq:stab}
 \stab(\cC) \simeq \colim( \cC \xrightarrow{\Sigma} \cC \xrightarrow{\Sigma} \ldots )
\end{equation}
as the colimit over iterations of the suspension functor on $\cC$.
The counit of the adjunction is typically denoted by $\Sigma^\infty$.

Recall the localization functor $\ell \colon \Waldone \to \catrex$ from \eqref{vrtojp4gf4rvvfeveve}.

\begin{ddd}\label{efgieofewfefewfwefwefewfwef}
 Define the \emph{universal localizing invariant for homotopical Waldhausen categories} to be the functor
 \[ \wuloc \colon \Waldone \xrightarrow{\ell} \catrex \xrightarrow{\stab} \catex \xrightarrow{\uloc} \mloc. \qedhere \]
\end{ddd}

The goal of the present section is to show that $\wuloc$ satisfies the obvious analogs of the fundamental theorems of connective Waldhausen $K$-theory.
For the most part, these are direct consequences of localization results from \cite[Chapter~7]{CisBook}.

\begin{prop}\label{prop:filtered-colimits}
 Let $I$ be a filtered poset. Then the canonical comparison map
 \[ \colim_I \circ \ell \to \ell \circ \colim_I \colon \Fun(I, \Waldone) \to \catrex \]
 is an equivalence of functors.
\end{prop}
\begin{proof}
 By \cite[Proposition~5.5.7.11]{LurieHTT}, filtered colimits in $\catrex$ may be computed in $\catinf$,
 so it is enough to consider $\ell' \colon \Waldone \to \catinf$, see \eqref{vewrverwefwwefwfew}.

 In the following argument we use the explicit model $ R\circ L\circ u$ of the functor $\ell^{\prime}$ given in  {\cref{rem:explicit-localization}}.
 Then we must check that the natural morphism
 \begin{equation}\label{ferfou43hfio3f43ff34f43f}
  \colim_{I} \circ Q\circ R_{I}\circ L_{I}\circ u_{I}\to R\circ L\circ u\circ \colim_{I}
 \end{equation} 
 is a weak equivalence in the marked model category structure on $\sSet^{+}$,
 where $Q$ is the cofibrant replacement in the projective model category structure on $\Fun(I,\sSet^{+})$,
 and the subscript $(-)_{I}$ stands for point-wise application of the corresponding functor.
  
 By inspection, $u \colon \Waldone \to \RelCatone$ preserves filtered colimits,
 and so does $L \colon \RelCatone \to \sSet^+$ because the nerve preserves filtered colimits. 
 Consequently  we have a factorization
 \begin{equation}\label{fewrfofoif2f3f32f2f32}
  \colim_{I} \circ Q \circ R_{I} \circ L_{I} \circ u_{I} \to R \circ \colim_{I} \circ L_{I} \circ u_{I} \xrightarrow{\simeq} R \circ L \circ u \circ \colim_{I}
 \end{equation}
 of \eqref{ferfou43hfio3f43ff34f43f}.
 We now consider the following commutative diagram
 \[\xymatrix{
   \colim_{I} \ar[r] \ar[d]^-{(1)} &  R\circ \colim\ar[d]^-{(3)} \\
   \colim_{I} \circ R_{I}\ar[r]^-{(2)}\ar[ur]^-{!} & R \circ \colim_{I} \circ R_{I} \\
 }\]
 As demonstrated in the proof of \cite[Proposition~3.1.3.7]{LurieHTT},
 marked equivalences are preserved under filtered colimits.
 Consequently, the morphisms $(1)$ and $(3)$ are weak equivalences.
 From the explicit description of fibrancy in $\sSet^+$ we deduce that a filtered colimit of fibrant objects is fibrant.
 Hence $(2)$ is a weak equivalence.
 Consequently, the morphism marked by $!$ is a weak equivalence, too.
 From \eqref{fewrfofoif2f3f32f2f32} we therefore get the factorization
 \[ \colim_{I}\circ  Q\circ R_{I}\circ L_{I}\circ u_{I}\to \colim_{I}\circ R_{I}\circ  L_{I}\circ u_{I} \xrightarrow{\simeq} R \circ L \circ u \circ \colim_{I}\ .\]
 We now again use that marked equivalences are preserved under filtered colimits in order to deduce that the first arrow is a weak equivalence, too.
\end{proof}

\begin{kor}\label{prop:k-filtered-colimits}
 The functor $\wuloc$ commutes with filtered colimits.
\end{kor}
\begin{proof}
 $\uloc$  commutes with filtered colimits by definition, and $\stab$ commutes with filtered colimits since it is a left adjoint.
 Hence the corollary is a consequence of \cref{prop:filtered-colimits}.
\end{proof}

Let $f \colon \bC \to \bD$ be an exact functor between homotopical Waldhausen categories.
\begin{ddd}\label{geroiger334g43g34g}
 The functor  $f \colon \bC \to \bD$ satisfies the \emph{approximation property} if the following conditions are satisfied:
 \begin{enumerate}
  \item $f$ preserves and detects weak equivalences;
  \item\label{geiogergerg43t3t34t} For every object $C$ of $\bC$ and morphism $y \colon f(C) \to D$ in $\bD$,
   there exist a morphism $x \colon C \to C'$ in $\bC$, a weak equivalence $w \colon f(C') \xrightarrow{\sim} D'$ and a weak equivalence $v \colon D \xrightarrow{\sim} D'$ such that the diagram
   \[\xymatrix{
    f(C)\ar[r]^-{y}\ar[d]_-{f(x)} & D\ar[d]^-{v}_-{\sim} \\
    f(C')\ar[r]_-{w}^-{\sim} & D'
   }\]
   commutes. \qedhere
 \end{enumerate} 
\end{ddd}

Recall the functor $\ell$ from \eqref{vrtojp4gf4rvvfeveve}.
\begin{theorem}[Cisinski Approximation theorem]\label{thm:approximation}
 If $f \colon \bC \to \bD$ is an exact functor between homotopical Waldhausen categories satisfying the approximation property,
 then $\ell(f) \colon \ell(\bC) \to \ell(\bD)$ is an equivalence.
\end{theorem}
\begin{proof}
 This is a special case of the dual of \cite[Proposition~7.6.15]{CisBook}.
\end{proof}

 As a first application of \cref{thm:approximation}, we establish a criterion to decide whether full inclusions of homotopical Waldhausen categories induce a fully faithful functor in $\catrex$.
This is the first step in proving the Cofinality theorem, which is our next goal.

\begin{ddd}\label{def:waldsubcat}
 A subcategory $\bC\subseteq \bD$  of the homotopical Waldhausen category $\bD$
 is a \emph{homotopical Waldhausen subcategory} if $(\bC, \bC \cap \cof\bD, \bC \cap w\bD)$ is also a homotopical Waldhausen category.
\end{ddd}

\begin{rem}
 Note that being a homotopical Waldhausen subcategory is a stronger assumption than
 being a `subcategory with cofibrations and weak equivalences' in the sense of \cite[page~321 and 327]{Wald85} which is in addition homotopical: If $\bC \subseteq \bD$ is a subcategory with cofibrations, then   $\cof\bC \subseteq \bC \cap \cof\bD$,
 whereas \cref{def:waldsubcat} requires $\cof\bC = \bC \cap \cof\bD$.
\end{rem}

\begin{ddd}\label{def:mca}
 The inclusion $\bC \subseteq \bD$ \emph{admits a mapping cylinder argument} if
 for every morphism $C\to D$ in $\bD$ such that $C$ belongs to $\bC$ and  $D$ is the target of a weak equivalence from an object of $\bC$ there exists a factorization
 \[\xymatrix@R=1em{
  & C^{\prime }\ar[dr]^-{\sim}& \\
  C\ar[ur]\ar[rr] & & D
 }\]
 with $ C^{\prime}$ in $\bC$.
\end{ddd}

Note that the weak equivalence $C^{\prime}\stackrel{\sim}{\to} D$ is not necessarily the one in the assumption on $D$.
A similar condition has been considered by Sagave~\cite[Theorem~3.1]{Sagave04}.

\begin{rem}
 The name of the property defined in \cref{def:mca} comes from the observation that this property is typically verified using a mapping cylinder construction, see for example the proof of \cref{cor:a-homology-coarsely-excisive}.
\end{rem}

\begin{rem}
 If $\bC \subseteq \bD$ admits a mapping cylinder argument, then the following seemingly stronger condition is also satisfied:
 Every diagram
 \[ C_1 \xrightarrow{y} D \xleftarrow[\sim]{w}  C_2 \]
 in $\bD$, in which $C_1$ and $C_2$ are objects in $\bC$,
 can be extended to a commutative diagram of the shape
 \begin{equation}\label{wefwef32r23r23}
  \xymatrix@R=1em{
   & D & \\
   C_1\ar[ur]^-{y}\ar[dr] & & C_2\ar[ul]_{w}^{\sim}\ar[dl]_{\sim} \\
   & C^{\prime}\ar[uu]^{\sim} &
  }
 \end{equation} 
 in which $C^{\prime}$ is an object in $\bC$.
 This follows from \cref{def:mca} by applying it to the morphism
 \[ C_1 \vee C_2 \xrightarrow{y  {+} w} D \]
 and rewriting the resulting diagram
 \[\xymatrix@R=1em{
  & C^{\prime}\ar[dr]^-{\sim} & \\
  C_{1} \vee C_{2} \ar[ur]\ar[rr]^-{y  {+} w} & & D
 }\]
 in the form \eqref{wefwef32r23r23}.
\end{rem}
 Let $\bC \subseteq \bD$ be a full homotopical Waldhausen subcategory of the homotopical Waldhausen category $\bD$.
 Recall the functor $\ell$ from \eqref{vrtojp4gf4rvvfeveve}.
\begin{prop}\label{lem:localization-fully-faithful}
 If the inclusion $\bC \subseteq \bD$ admits a mapping cylinder argument, then $\ell(\bC) \to \ell(\bD)$ is fully faithful.
\end{prop}
\begin{proof}
 The proof consists of the following steps:
 \begin{enumerate}
  \item\label{34t3t34rgergergerg} We consider the saturation $\overline{\bC}$ of $\bC$ in $\bD$ defined as the full subcategory consisting of all objects of $\bD$
   which are connected by a sequence of zig-zags of weak equivalences with an object of $\bC$.
   We first show in \cref{gbioregreggergreg} that $\overline{\bC}$ is a homotopical Waldhausen subcategory of $\bD$.
   \item We then show in \cref{erguiergrgegg}, using \cref{thm:approximation}, that the inclusion $\bC \to \overline{\bC}$ induces
   an equivalence $\ell(\bC) \xrightarrow{\simeq} \ell(\overline{\bC})$ of $\infty$-categories.
  \item Finally, we show in \cref{rfiuwefwefewf23r23r2r} that the functor $\ell(\overline{\bC}) \to \ell(\bD)$ induced by the inclusion $\overline{\bC} \to \bD$ is fully faithful.
 \end{enumerate}
 This implies the assertion since $\ell(\bC) \to \ell(\bD)$ is the composition of an equivalence and a fully faithful functor.
\end{proof}

As a preparation, we show that a sequence of zig-zags of weak equivalences in $\bD$ can be reduced to a single inward-pointing one.
\begin{lem}\label{htrtoihjiorgergergerge}
 If two objects $D$ and $D^{\prime}$ in $\bD$ are connected by a sequence of zig-zags of weak equivalences,
 then there exists a zig-zag of weak equivalences
 \[ D \xrightarrow{\sim} D'' \xleftarrow{\sim} D'\ .\]
 \end{lem}
\begin{proof}
Cf.~\cite[Lemma~5.10]{BM11} for a very similar argument.

 We claim that we can replace outward pointing zig-zags of weak equivalences by inward-pointing zig-zags of weak equivalences.
 If $D$ and $D^{\prime}$ in $\bD$ are connected by a sequence of zig-zags of weak equivalences, then  we can
 apply the claim repeatedly to parts of the sequences and finally compose the maps in order to obtain the desired single inward-pointing zig-zag of weak equivalences connecting $D$ and $D^{\prime}$.

 We now show the claim. We consider an outward pointing zig-zag
 \[ D \xleftarrow[\sim]{x} D'' \xrightarrow[\sim]{x'} D' \]
 of weak equivalences in $\bD$.
 Then we choose a factorization
 \[ D \vee D'' \rightarrowtail D''' \xrightarrow{\sim} D \]
 of the morphism $D \vee D''\xrightarrow{\id + x} D$ into a cofibration followed by a weak equivalence.
 We have a commuting diagram
 \[\xymatrix{
  D \vee D'' \ar@{>->}[r]&D''' \ar[d]^{\sim} \\
  D\ar@{=}[r]\ar@{>->}[u]\ar@{>..>}[ur]^{y}_{\sim} & D
 }\]
 which shows that the morphism $y$ is a cofibration and a weak equivalence.
 Similarly, the commuting diagram
 \[\xymatrix{
  D \vee D'' \ar@{>->}[r]&D'''  \ar[d]^{\sim} \\
  D''\ar[r]^{\sim}_{x}\ar@{>->}[u]\ar@{>..>}[ur]^{z}_{\sim} & D
 }\]
 shows that $z$ is a cofibration and a weak equivalence.
 Since $z$ is cofibration, we can form a push-out square
 \[\xymatrix{
  D''\ar[r]^{x'}_{\sim}\ar@{>->}[d]_{z}^{\sim} & D'\ar@{>->}[d] \\
  D'''\ar[r] & D'''\sqcup_{D''} D'\ .
 }\]
 The lower horizontal arrow is a weak equivalence because it is the pushout of a weak equivalence.
 We now see that
 \[ D \xrightarrow{\sim} D''' \xrightarrow{\sim} D''' \sqcup_{D''} D' \xleftarrow{\sim} D' \]
 is the desired inward-pointing  zig-zag of weak equivalences by composing the first two morphisms.
 \end{proof}

Let $\overline{\bC}$ be the {saturation} of $\bC$ in $\bD$, i.e.~the full subcategory consisting of all objects of $\bD$
which are connected by a sequence of zig-zags of weak equivalences with an object of $\bC$.
\begin{lem}\label{gbioregreggergreg}
  If $\bC \subseteq \bD$ admits a mapping cylinder argument, then $\overline{\bC}$ is a full homotopical Waldhausen subcategory of $\bD$.
 \end{lem}
 \begin{proof}
  Since $\overline{\bC}$ obviously admits factorizations, we must only show that $\overline{\bC}$ is closed under pushouts along cofibrations.
  
 Let the diagram
 \[ \overline{C}_2 \leftarrow \overline{C}_0 \rightarrowtail \overline{C}_1 \]
 in $\overline{\bC}$ be given.  In the following, we describe step by step how to construct the diagram
 \[\xymatrix@R=3em{
  \overline{C}_1 &  \overline{C}_1\ar@{=}[l]_-{\id}\ar[r]^-{\sim} &  \overline{C}_1 \sqcup_{\overline{C_0}} D_0 & \overline{C}_{1}' \ar[l]_-{\sim}\ar@{=}[r]^-{\id} &  \overline{C}'_{1}&  \overline{C}'_{1}\ar@{=}[l]_-{\id} &  \overline{C}'_{1}\ar@{=}[l]_-{\id}  \\
   \overline{C}_0\ar@{>->}[u]\ar[d] &  \overline{C}_0\ar @{>->}[u]\ar@{>->}[d]\ar@{=}[l]_-{\id}\ar[r]^-{\sim} & D_0\ar@{>->}[u]\ar@{>->}[d] & C_0\ar@{>->}[u]\ar@{>->}[d]\ar[l]_-{\sim}\ar@{=}[r]^-{\id} & C_0\ar@{>->}[u]\ar[d] & C_0\ar@{>->}[u]\ar[d]\ar@{=}[l]_-{\id} & C_0\ar@{>->}[u]\ar@{>->}[d]\ar@{=}[l]_-{\id} \\
   \overline{C}_2 & \overline{C}'_{2}\ar[l]_-{\sim}\ar[r]^-{\sim} & \overline{C}_{2}' \sqcup_{ \overline{C}_0} D_0 & \overline{C}''_{2}\ar[l]_-{\sim}\ar[r]^-{\sim} & D_2 & C'_2\ar[l]_-{\sim} & C_2\ar[l]_-{\sim}
 }\]
 
 The first column is the original span.
 
 The second column is obtained from the first by factoring the morphism $ \overline{C}_0 \to  \overline{C}_2$ into a cofibration followed by a weak equivalence.
 By \cref{htrtoihjiorgergergerge}, there exists a zig-zag
 \[ \overline{C}_0 \xrightarrow{\sim} D_0 \xleftarrow{\sim} C_0 \]
 of weak equivalences with $C_0 \in \bC$.
 
 Taking the indicated pushouts produces the third column of the diagram.
 
 The fourth column arises from the third by factoring the morphisms $C_0 \to  \overline{C}_1 \sqcup_{\overline{C}_0} D_0$ and $C_0 \to \overline{C}_{2}' \sqcup_{\overline{C}_0} D_0$
 into a cofibration followed by a weak equivalence.
 Again using \cref{htrtoihjiorgergergerge} and the fact that $\overline{C}_{2}'' \in \overline{\bC}$, we choose a zig-zag
 \[ \overline{C}''_2 \xrightarrow{\sim} D_2 \xleftarrow{\sim} C \]
 of weak equivalences with $C \in \bC$.
 
 The fifth column arises canonically by composing existing morphisms.
 
 For the next column, we  {use that $\bC \subseteq \bD$ admits a mapping cylinder argument} (note that $D_2$ is the target of a weak equivalence from the object $C $ of $\bC$)
 in order to factor the morphism $C_0 \to D_2$ as
 \[ C_0 \to C'_2 \xrightarrow{\sim} D_2 \]
 with $C'_{2}\in \bC$.

 The final column is then obtained by factoring the morphism $C_0 \to C'_2$ in $\bC$ into a cofibration followed by a weak equivalence; in particular, we have $C_2 \in \bC$.
 
 Note that the final column has the same shape as the fourth column,
 with the difference that $C_0$ and $C_2$ both lie in $\bC$.
 Repeating the argument that gave the right half of the diagram with $\overline{C}'_{1}$ in place of $\overline{C}''_2$ finally yields a zig-zag of weak equivalences
 between the original span and a span in $\bC$.
 By virtue of the gluing axiom in a Waldhausen category, it follows that $\overline{C}_2 \sqcup_{\overline{C}_0} \overline{C}_1$ lies in $\overline{\bC}$.
\end{proof}
 
 \begin{lem}\label{erguiergrgegg}
 If $\bC \subseteq \bD$ admits a mapping cylinder argument, then the inclusion $\bC\to \overline{\bC}$  induces an equivalence  $\ell(\bC)\to \ell(\overline{\bC})$.
 \end{lem}
 \begin{proof}
  {We verify that the assumptions of \cref{thm:approximation} are satisfied.
  The saturation $(\overline{\bC}, \overline{\bC} \cap \cof\bD, \overline{\bC} \cap w\bD)$ of $\bC$ in $\bD$ is a homotopical Waldhausen category by \cref{gbioregreggergreg}.}
  It remains to show that the inclusion $\bC\to \overline{\bC}$ has the approximation property, see \cref{geroiger334g43g34g}.
 
  The inclusion of course preserves and detects equivalences. 
  
  In order to verify the second condition \ref{geroiger334g43g34g}.\ref{geiogergerg43t3t34t}, we consider a morphism
  $y \colon C \to \overline{C}$ with $C \in \bC$ and $\overline{C} \in \overline{\bC}$.
  By \cref{htrtoihjiorgergergerge}, we find a zig-zag of weak equivalences in $\bD$
  \[ \overline{C} \xrightarrow{\sim} \overline{C}^{\prime}  \xleftarrow{\sim} C'' \]
  with $C''\in \bC$ and  $\overline{C}^{\prime} \in \overline{\bC}$.
  In particular, $\overline{C}^{\prime} $ is the target of a weak equivalence from an object of $\bC$
  and we can  {use that $\bC \subseteq \bD$ admits a mapping cylinder argument} in order to factorize the  {composed map
  $C \to \overline{C}^{\prime}$ as $C \to C' \xrightarrow{\sim} \overline{C}^{\prime}$} with $C^{\prime}\in \bC$.
  This determines the desired commutative square
  \[\begin{gathered}[b]
  \xymatrix{
   C\ar[r]\ar[d] & \overline{C}\ar[d]^-{\sim} \\
   C'\ar[r]^{\sim} & \overline{C}^{\prime}
  }\\ \end{gathered}\qedhere\]
 \end{proof}

 \begin{lem}\label{rfiuwefwefewf23r23r2r}
  {If $\bC \subseteq \bD$ admits a mapping cylinder argument, then the} inclusion $\overline{\bC}\to \bD$ induces a fully faithful functor $\ell(\overline{\bC})\to \ell(\bD)$.
 \end{lem}
\begin{proof}
 $\overline{\bC}$ is a homotopical Waldhausen category by \cref{gbioregreggergreg}.

 Noting that the opposite of a homotopical Waldhausen category is a category with weak equivalences and fibrations,
 we apply dualized versions of the results in \cite[Chapter~7]{CisBook}.
 
 \cite[Example~7.6.4]{CisBook} shows that there exists a \emph{calculus of fractions} at each object of $\bD$.
 This allows us to apply \cite[Corollary~7.2.9]{CisBook}:
 Let $l \colon \bD \to \ell(\bD)$ be the localization functor.
 For two objects $D_1$ and $D_2$ in $\bD$, we have a canonical equivalence
 \[ \Map_{\ell(\bD)}(l(D_1),l(D_2)) \simeq \colim_{D_2 \xrightarrow{\sim} D_2' \in W_\bD(D_2)} \Map_{\bD}(D_1,D_2'), \]
 where the colimit is indexed by the full subcategory $W_{\bD}(D_2)$ of the comma category $D_2/\bD$ spanned by the weak equivalences with domain $D_2$.
 
 Applying the same formula for $\overline{\bC}$, the inclusion functor $\overline{\bC}\to \bD$ induces for every two objects $\overline{C}_1$, $\overline{C}_2$ in $\overline{\bC}$ the canonical map
 \[ \colim_{(\overline{C}_2 \xrightarrow{\sim} \overline{C}_2') \in W_{\overline{\bC}}(\overline{C}_2)} \Map_{\overline{\bC}}(\overline{C}_1,\overline{C}_2') \to \colim_{(\overline{C}_2 \xrightarrow{\sim} D) \in W_{\bD}(\overline{C}_2)} \Map_{\bD}(\overline{C}_1,D)\ . \]
 Since $\overline{\bC}$ is a full subcategory of $\bD$ which is closed under weak equivalences,
 both sides are colimits of the same diagram.
 Hence $\ell(\overline{\bC}) \to \ell(\bD)$ is fully faithful.
\end{proof}

 This finishes the proof of \cref{lem:localization-fully-faithful}. 

 Recall the idempotent completion functor
 \[ \Idem \colon \catinf \to \catinf \]
 from \cite[Proposition~5.4.2.18]{LurieHTT}.
 Denote by $\catperf$ the full subcategory of $\catex$ spanned by the idempotent complete stable $\infty$-categories,
 and let $\catrperf$ be the full subcategory of $\catrex$ spanned by the idempotent complete right-exact $\infty$-categories.

\begin{lem}\label{lem:idem-exact}
 The idempotent completion functor induces functors
 \[ \Idem^{\mathrm{Rex}} \colon \catrex \to \catrperf \quad\text{and}\quad \Idem^{\mathrm{ex}} \colon \catex \to \catperf \]
 which are left adjoint to the inclusions $\catrperf \to \catrex$ and $\catperf \to \catex$. respectively.
\end{lem}
\begin{proof}
 We begin by showing that $\Idem(\cC)$ is right-exact for any $\cC \in \catrex$.
 The argument is very similar to the proof of \cite[Corollary~1.1.3.7]{LurieHA}, but easier.
 
 Let $\kappa$ be a regular cardinal and consider the canonical functor $j \colon \cC \to \Ind_\kappa(\cC)$.
 By \cite[Proposition~5.1.3.2]{LurieHTT} and \cite[Proposition~5.3.5.14]{LurieHTT}, $\Ind_\kappa(\cC)$ has a zero object and $j$ is exact.
 Using \cite[Proposition~5.3.5.15]{LurieHTT}, every finite diagram in $\Ind_\kappa(\cC)$ is a $\kappa$-filtered colimit of finite diagrams in $\cC$.
 Since $\cC$ admits all finite colimits and colimits commute with each other, this shows that $\Ind_\kappa(\cC)$ is right-exact.
 
 By \cite[Lemma~5.4.2.4]{LurieHTT}, $\Idem(\cC)$ is equivalent to the full subcategory of $\kappa$-compact objects in $\Ind_\kappa(\cC)$.
 Since zero objects are compact and the collection of $\kappa$-compact objects is closed under finite colimits,
 this shows that $\Idem(\cC)$ is right-exact and that the canonical functor $\cC \to \Idem(\cC)$ is exact.
 
 Suppose now that $f \colon \cC \to \ccD$ is an exact functor, where $\ccD \in \catrex$ is idempotent complete.
 Another application of \cite[Proposition~5.3.5.15]{LurieHTT} shows that the induced functor $\Ind_\kappa(\cC) \to \Ind_\kappa(\ccD)$ preserves finite colimits.
 Restricting to the full subcategories of $\kappa$-compact objects and invoking \cite[Lemma~5.4.2.4]{LurieHTT} again,
 we see that the induced functor $\Idem(\cC) \to \ccD$ is exact.
 
 Hence the idempotent completion functor restricts to a functor
 \[ \Idem^{\mathrm{Rex}} \colon \catrex \to \catrperf\ .\]
 The universal property of $\Idem(\cC)$ implies that the canonical functor $\cC \to \Idem(\cC)$ induces an equivalence 
 \[ {\Map_{\catrex}(\Idem(\cC), \ccD) \xrightarrow{\sim} \Map_{\catrex}(\cC, \ccD)}\ , \]
 exhibiting $\Idem^{\mathrm{Rex}}$ as a left adjoint of the inclusion functor $\catrperf \to \catrex$.
 
 Since the idempotent completion of stable $\infty$-category is stable by \cite[Corollary~1.1.3.7]{LurieHA},
 it follows from \cite[Proposition~1.1.4.1]{LurieHA} that $\Idem^{\mathrm{Rex}}$ further induces a functor
 \[ \Idem^{\mathrm{ex}} \colon \catex \to \catperf \] 
 which is left adjoint to the inclusion $\catperf \to \catex$.
\end{proof}

Let $\bC\subseteq \bD$ be a full homotopical Waldhausen subcategory and let $D$ be an object of $\bD$.

\begin{ddd}\label{ergiofrggergrge}
 We say that $D$ is \emph{dominated by $\bC$} if there exists 
 a diagram
 \[\xymatrix@C=2em{
  D'\ar[dr]\ar[rr]^-{\sim} & & D \\
  & C\ar[ur] & \\
 }\]
 in $\bD$ with $C $ in $\bC$.

 We further say that \emph{$\bD$ is dominated by $\bC$ }if every object of $\bD$ is dominated by $\bC$.
\end{ddd}

\begin{theorem}[Cofinality theorem]\label{thm:cofinality}
 Let $\bC \subseteq \bD$ be the inclusion of a full homotopical Waldhausen subcategory.
 Suppose that
 \begin{enumerate}
  \item\label{it:cofinality-1} the inclusion $\bC \subseteq \bD$ admits a mapping cylinder argument, see \cref{def:mca};
  \item\label{it:cofinality-2} $\bD$ is dominated by $\bC$, see \cref{ergiofrggergrge}.
 \end{enumerate}
 Then the inclusion $\bC \to \bD$ induces an equivalence $\Idem(\ell(\bC)) \to \Idem(\ell(\bD))$.
\end{theorem}
\begin{proof}
 By \cref{lem:localization-fully-faithful} and assumption~\eqref{it:cofinality-1}, the induced functor $\ell(\bC) \to \ell(\bD)$ is fully faithful.
 Assumption~\eqref{it:cofinality-2} implies that every object in $\ell(\bD)$ is a retract of an object in $\ell(\bC)$.
 Hence, the induced functor $\Idem(\ell(\bC)) \xrightarrow{\sim} \Idem(\ell(\bD))$ is an equivalence. 
\end{proof}

\begin{lem}\label{lem:stab-idem}
 The stabilization functor $\stab \colon \catrex \to \catex$ restricts to a functor
 \[ \stab^{\mathrm{perf}} \colon \catrperf \to \catperf \]
 which is left adjoint to the inclusion $\catperf \to \catrperf$.
\end{lem}
\begin{proof}
 Let $\cC$ be an idempotent complete, right-exact $\infty$-category.
 Since $\Idem^{\mathrm{Rex}}$ is a left adjoint functor, formula~\eqref{eq:stab} implies
 \[ \stab(\cC) \simeq \stab(\Idem(\cC)) \simeq \Idem(\stab(\cC))\ , \]
 so $\stab$ preserves idempotent completeness as claimed.
\end{proof}
\begin{kor}\label{lem:idem-and-stab}
There is a canonical equivalence
 \[ \stab^{\mathrm{perf}} \circ \Idem^{\mathrm{Rex}} \simeq \Idem^{\mathrm{ex}} \circ \stab \]
 of functors from $\catrex$ to $\catperf$.
\end{kor}
\begin{proof}
 Since the various inclusion functors form a commutative square
 \[\xymatrix{
  \catrperf\ar[r]\ar[d] & \catrex\ar[d] \\
  \catperf\ar[r] & \catex
 }\]
 the equivalence of functors $\stab^{\mathrm{perf}} \circ \Idem^{\mathrm{Rex}} \simeq \Idem^{\mathrm{ex}} \circ \stab$ is a formal consequence of \cref{lem:idem-exact} and \cref{lem:stab-idem}.
\end{proof}

\begin{kor}
 Let $\bC \subseteq \bD$ be a full homotopical Waldhausen subcategory satisfying the assumptions of \cref{thm:cofinality}.
 Then the induced map $\wuloc(\bC) \to \wuloc(\bD)$ is an equivalence.
\end{kor}
\begin{proof}
 Since $\uloc$ sends the units $\cC\to \Idem(\cC)$ to equivalences by definition
 and $\Idem$ commutes with $\stab$ by \cref{lem:idem-and-stab},
 this follows from \cref{thm:cofinality}.
\end{proof}

To conclude, we discuss the Fibration theorem.
The Additivity theorem will follow from this as a corollary.

\begin{lem}\label{lem:stabilization-fully-faithful}
 If $f \colon \cC \to \ccD$ is a fully faithful functor in $\catrex$,
 then its stabilization $\stab(f) \colon \stab(\cC) \to \stab(\ccD)$ is also fully faithful.
\end{lem}
\begin{proof}
 {We use the colimit description \eqref{eq:stab} of $\stab(\cC)$ and $\stab(\ccD)$, respectively.
 Let $i_n \colon \cC \to \stab(\cC)$ and $j_n \colon \ccD \to \stab(\ccD)$ denote the $n$-th structure morphism of the colimit systems.}
 For any two objects $\overline X$ and $\overline Y$ of $\stab(\cC)$, there exist some natural number $n$ and objects $X, Y \in \cC$ satisfying $\overline X \simeq i_n(X)$ and $\overline Y \simeq i_n(Y)$.
 Since $\stab(f)(\overline X) \simeq j_n(f(X))$ and using that $f$ is exact and fully faithful, we have
 \begin{align*}
  \Map_{\stab(\cC)}(\overline X, \overline Y)
  &\simeq \Map_{\stab(\cC)}(i_n(X), i_n(Y)) \\
  &\simeq \colim_k \Map_{\cC}(\Sigma^{k-n} X, \Sigma^{k-n} Y) \\
  &\simeq \colim_k \Map_{\ccD}(f(\Sigma^{k-n} X), f(\Sigma^{k-n} Y)) \\
  &\simeq \colim_k \Map_{\ccD}(\Sigma^{k-n} f(X), \Sigma^{k-n} f(Y)) \\
  &\simeq \Map_{\ccD}(j_n(f(X)), j_n(f(Y))) \\
  &\simeq \Map_{\stab(\ccD)}(\stab(f)(\overline X), \stab(f)(\overline Y)).
 \end{align*}
 So $\stab(f)$ is fully faithful.
\end{proof}

\begin{theorem}[Fibration theorem]\label{thm:fibration}
 Let $(\bC,v\bC)$ and $(\bC,w\bC)$ be homotopical Waldhausen categories having the same underlying category with cofibrations such that $v\bC \subseteq w\bC$.
 Denote by $\bC^w$ the full subcategory of $\bC$ spanned by those objects which are $w$-equivalent to zero.
 
 Then the inclusion $\bC^w \subseteq \bC$ and the identity on $\bC$ induce a fiber sequence
 \[ \wuloc(\bC^w,v\bC^w) \to \wuloc(\bC,v\bC) \to \wuloc(\bC,w\bC) \]
 in $\mloc$.
\end{theorem}
\begin{proof}
 Let us introduce some shorthand notation. Set
 \begin{align*}
  \bC_v := \ell(\bC, v\bC), &\quad \bC_w := \ell(\bC,w\bC), \\
  \bC^s_v := \stab(\ell(\bC,v\bC)), &\quad \bC^s_w := \stab(\ell(\bC,w\bC)).
 \end{align*}
 Since $v\bC \subseteq w\bC$, we have a commutative diagram
 \[\xymatrix{
  \bC\ar[r]^{l_v}\ar[rd]_{l_w} & \bC_v\ar[r]^{\Sigma^\infty_v}\ar[d]^{p} & \bC_v^s\ar[d]^{p^s} \\
   & \bC_w\ar[r]^{\Sigma^\infty_w} & \bC^s_w
 }\]
 in which $l_v$ and $l_w$ are the respective localization functors, $\Sigma^\infty_v$ and $\Sigma^\infty_w$ are instances of the counit $\Sigma^\infty$,
 and $p$ and $p^s$ are induced by the identity on $\bC$.
 
 Denote by $\bC_0 \subseteq \bC$ the full subcategory spanned by those objects whose image under $\Sigma^\infty_w \circ l_w$ is zero.
 Since $l_w$ is exact by \cref{prop:exact-localization}\eqref{it:exact-localization-1}, $\bC_0$ inherits Waldhausen structures $(\bC_0, v\bC_0 := \bC_0 \cap v\bC)$ and $(\bC_0, w\bC_0 := \bC_0 \cap w\bC)$.
 As above, we abbreviate $\bC_{0,v} := \ell(\bC_0,v\bC_0)$, and similarly for $\bC_{0,w}$, $\bC_{0,v}^s$ and $\bC_{0,w}^s$.
 
 Since $\bC_0$ is closed under $w$-equivalences and $v\bC \subseteq w\bC$,
 the inclusions $(\bC_0, v\bC_0) \subseteq (\bC,v\bC)$ and $(\bC_0, w\bC_0) \subseteq (\bC,w\bC)$ both admit a mapping cylinder argument (see \cref{def:mca}).
 It follows from \cref{lem:localization-fully-faithful} that the induced functors $\bC_{0,v} \to \bC_v$ and $\bC_{0,w} \to \bC_w$ are fully faithful.
 Hence $\bC_{0,v}^s \to \bC_v^s$ and $\bC_{0,w}^s \to \bC_w^s$ are also fully faithful by \cref{lem:stabilization-fully-faithful}.

 Recall from \cite[Theorem~I.3.3]{NS} the Verdier localization of stable $\infty$-categories:
 Given a stable $\infty$-category $\cC$ and a full stable subcategory $\ccD$,
 the localization $\cC/\ccD$ of $\cC$ at all morphisms whose cofiber lies in $\ccD$ is a stable $\infty$-category.
 For any stable $\infty$-category $\cE$, restriction along the localization functor $\cC \to \cC/\ccD$ induces an equivalence
 $\Fun^{\mathrm{ex}}(\cC/\ccD,\cE) \xrightarrow{\sim} \Fun^{\mathrm{ex}}_\ccD(\cC,\cE)$,
 where the right hand side denotes the full subcategory of those exact functors which vanish on $\ccD$.
 
 By the universal property, $p^s$ induces an exact functor
 \[ \overline{p} \colon \bC_v^s/\bC_{0,v}^s \to \bC_w^s/\bC_{0,w}^s \]
 on Verdier localizations. We show that $\overline{p}$ is an equivalence by constructing an inverse equivalence $\overline{q}$.
 
 Consider the functor
 \[ L_v \colon \bC \xrightarrow{l_v} \bC_v \xrightarrow{\Sigma^\infty_v} \bC_v^s \to \bC_v^s/\bC_{0,v}^s. \]
 Let $x \colon X \to Y$ be a morphism in $w\bC$, and denote by $\cofib_v(x) \in \bC_v$ the cofiber of $l_v(x)$.
 Since $p$ is exact, we have $p(\cofib_v(x)) \simeq 0$. Hence $\cofib_v(x)$ lies in the essential image of the functor $\bC_0 \subseteq \bC \xrightarrow{l_v} \bC_v$.
 It follows that $\cofib_v(x)$ is sent to zero in $\bC_v^s/\bC_{0,v}^s$.
 Since a morphism in a stable $\infty$-category is an equivalence if and only if its cofiber vanishes,
 this proves that $L_v$ sends all morphisms in $w\bC$ to equivalences.
 By the universal property of localization and stabilization, there exists an induced functor $q^s \colon \bC_w^s \to \bC_v^s/\bC_{0,v}^s$ which fits into the commutative diagram
 \[\xymatrix{
  \bC_{0,v}^s\ar[r]\ar[d] & \bC_v^s\ar[d]_{p^s}\ar[r] & \bC_v^s/\bC_{0,v}^s\\
  \bC_{0,w}^s\ar[r] & \bC_w^s\ar[ru]_{q^s} & \\
 }\]
 Since the functor $\bC_{0,v}^s \to \bC_{0,w}^s$ is essentially surjective, $q^s$ vanishes on $\bC_{0,w}^s$ and hence induces an exact functor
 \[ \overline{q} \colon \bC_w^s/\bC_{0,w}^s \to \bC_v^s/\bC_{0,v}^s. \]
 By construction, we have $\overline{q}\overline{p} \simeq \id$.
 
 To verify that $\overline{p}\overline{q} \simeq \id$ also holds, we show that $L_w \colon \bC \to \bC_w^s/\bC_{0,w}^s$, which is defined analogously to $L_v$, enjoys a universal property.
 Let $\ccD$ be a stable $\infty$-category.
 By the universal property of the Verdier localization, we have an equivalence
 \[ \Fun^{\mathrm{ex}}(\bC_w^s/\bC_{0,w}^s, \ccD) \xrightarrow{\sim} \Fun^{\mathrm{ex}}_{\bC_{0,w}^s}(\bC_w^s, \ccD). \]
 Restriction along $\Sigma^\infty_w$ induces a functor
 \[ (\Sigma^\infty_w)^* \colon \Fun^{\mathrm{ex}}_{\bC_{0,w}^s}(\bC_w^s, \ccD) \to \Fun^{\mathrm{ex}}_{\bC_{0,w}}(\bC_w, \ccD). \]
 Let $f \colon \bC_w \to \ccD$ be a functor vanishing on $\bC_{0,w}$,
 and let $\widehat{f} \colon \bC_w^s \to \ccD$ denote the essentially unique exact functor satisfying $\widehat{f} \circ \Sigma^\infty_w \simeq f$.
 Let $\overline{X} \in \bC_{0,w}^s$.
 Using the colimit description \eqref{eq:stab} of the stabilization,
 there exists some natural number $n$ such that $\overline{X} \simeq i_n(X)$ for some $X \in \bC_{0,w}^s$,
 where $i_n \colon \bC_{0,w} \to \bC_{0,w}^s$ is the $n$-th structure map of the colimit.
 By assumption, $\Sigma^n\overline{X}$ gets mapped to zero in $\ccD$. Since $\Sigma$ is an equivalence on $\bC_{0,w}^s$, it follows that $\widehat{f}(\overline{X}) \simeq 0$,
 so $\widehat{f}$ vanishes on $\bC_{0,w}^s$.
 This proves that the functor $(\Sigma^\infty_w)^*$ is an equivalence. 
 
 Finally, restriction along $l_w$ induces a functor
 \[ l_w^* \colon \Fun^{\mathrm{ex}}_{\bC_{0,w}}(\bC_w, \ccD) \to \Fun^{\mathrm{ex}}_{\bC_0}(\bC, \ccD), \]
 where $\Fun^{\mathrm{ex}}_{\bC_0}(\bC, \ccD)$ denotes the full subcategory of $\Fun^{\mathrm{ex}}(\bC, \ccD)$ containing those functors
 which send all objects in $\bC_0$ to zero (see \cref{def:exact-functor} for the notion of exactness of a functor $\bC \to \ccD$).
 Since the localization $\bC_0 \to \bC_{0,w}$ is essentially surjective, $l_w^*$ is also an equivalence.
 
 Consequently, $L_w$ induces an equivalence
 \[ \Fun^{\mathrm{ex}}(\bC_w^s/\bC_{0,w}^s, \ccD) \xrightarrow{\sim} \Fun^{\mathrm{ex}}_{\bC_0}(\bC, \ccD) \]
 for any stable $\infty$-category $\ccD$.
 Choosing $\ccD = \bC_w^s/\bC_{0,w}^s$ and tracing through the definitions,
 we find that $\overline{p}\overline{q}$ corresponds to $L_w$ under this equivalence.
 Therefore, $\overline{p}\overline{q} \simeq \id$, so $\overline{p}$ is an equivalence.
 Since we have a commutative diagram
 \[\xymatrix{
  \bC_{0,v}^s\ar[r]\ar[d]^{\id} & \bC_v^s\ar[r]\ar[d]^{\id} & \bC_v^s/\bC_{0,v}^s\ar[d]^{\overline{p}}_{\sim} \\
  \bC_{0,v}^s\ar[r] & \bC_v^s\ar[r] & \bC_w^s/\bC_{0,w}^s
 }\]
 the lower line is a Verdier sequence.
 
 Consider now the Verdier sequence
 \[ \bC_{0,w}^s \to \bC_w^s \to \bC_w^s/\bC_{0,w}^s. \]
 By definition, the first map in this sequence is zero, so $\bC_{0,w}^s \simeq 0$.
 It follows that $\bC_w^s \to \bC_w^s/\bC_{0,w}^s$ is an equivalence, and hence
 \[ \bC_{0,v}^s \to \bC_v^s\xrightarrow{p^s} \bC_w^s \]
 is a Verdier sequence.
 
 Observe that $\bC^w \subseteq \bC_0$. We next claim that the induced functor
 \[ \stab(\ell(\bC^w,v\bC^w)) \to \bC_{0,v}^s \]
 is an equivalence.
 
 Since $\bC^w$ is closed under $v$-equivalences, the inclusion $\bC^w \subseteq \bC_0$ admits a mapping cylinder argument.
 It follows from \cref{lem:localization-fully-faithful} and \cref{lem:stabilization-fully-faithful} that $\stab(\ell(\bC^w,v\bC^w)) \to \bC_{0,v}^s$ is fully faithful.
 
 Let $X$ be an object in $\bC_0$, which is equivalent to saying that $\Sigma^\infty_w(l_w(* \to X))$ is an equivalence.
 Using the colimit description \eqref{eq:stab} of stabilization,
 this implies that $0 \to \Sigma^nl_w(X)$ is an equivalence in $\bC_w$ for some natural number $n$.
 Since $p$ is exact, we have $p(\Sigma^nl_v(X)) \simeq \Sigma^nl_w(X)$ in $\bC_w$.
 As $l_v$ is essentially surjective, there exists some $Y \in \bC$ such that $l_v(Y) \simeq \Sigma^nl_v(X)$.
 In particular, $p(0 \to l_v(Y))$ is an equivalence.
 Since $p \circ l_v \simeq l_w$ and $l_w$ detects weak equivalences by \cref{prop:exact-localization}\eqref{it:exact-localization-1},
 we conclude that $Y$ is weakly contractible with respect to $w\bC$, i.e.~$Y \in \bC^w$.
 
 Therefore, some iterated suspension of each object in $\bC_{0,v}$ lies in the essential image of the functor $\ell(\bC^w,v\bC^w) \to \ell(\bC_0,v\bC_0)$.
 Hence $\stab(\ell(\bC^w,v\bC^w)) \to \stab(\ell(\bC_0,v\bC_0))$ is essentially surjective.
 
 We conclude that
 \[ \stab(\ell(\bC^w,v\bC^w)) \to \stab(\ell(\bC,v\bC)) \xrightarrow{p^s} \stab(\ell(\bC,w\bC)) \]
 is a Verdier sequence.
 The theorem is now a consequence of \cite[Theorem~9.34]{BGT13} and \cite[Proposition~I.3.5]{NS}.
\end{proof}

\begin{kor}[Additivity theorem]\label{thm:additivity}
 The projection functor
 \[ (s,q) \colon S_2\bC \to \bC \times \bC,\quad (X \rightarrowtail Y \twoheadrightarrow Z) \mapsto (X,Z) \]
 induces an equivalence
 \[ \wuloc(S_2\bC) \xrightarrow{\sim} \wuloc(\bC) \oplus \wuloc(\bC). \]
\end{kor}
\begin{proof}
 Consider the subcategory of weak equivalences $w_qS_2\bC$ consisting of those morphisms whose image under $q$ is a weak equivalence.
 By \cref{thm:fibration}, we obtain a fiber sequence
 \[ \wuloc(S_2\bC^{w_q},w) \to \wuloc(S_2\bC,w) \to \wuloc(S_2\bC,w_q). \]
 The projection functor $q \colon S_2\bC \to \bC$ admits the section
 \[ \bC \to S_2\bC,\quad Z \mapsto (* \rightarrowtail Z \twoheadrightarrow Z). \]
 If we consider the Waldhausen structure $(S_2\bC,w_q)$,
 this section is an inverse up to weak equivalence, so  {$q$ induces an equivalence $\wuloc(S_2\bC,w_q) \xrightarrow{\sim} \wuloc(\bC)$.}
 
 An object $X \rightarrowtail Y \twoheadrightarrow Z$ lies in $S_2\bC^{w_q}$ precisely if $Z$ is  weakly contractible.
 Hence, the functor $i \colon \bC \to S_2\bC^{w_q}$ sending $X$ to $X \rightarrowtail X \twoheadrightarrow *$ is a right-inverse to $s|_{S_2\bC^{w_q}}$ and a left-inverse up to weak equivalence.
 It follows that $i$ induces an equivalence $\wuloc(\bC) \xrightarrow{\sim} \wuloc(S_2\bC^{w_q})$.
 
 Therefore, we obtain a fiber sequence
 \[ \wuloc(\bC) \xrightarrow{\wuloc(i)} \wuloc(S_2\bC) \xrightarrow{\wuloc(q)} \wuloc(\bC). \]
 Since $i$ is split by $s$ and $\mloc$ is stable, the claim follows.
\end{proof}

\subsection{Algebraic \texorpdfstring{$K$}{K}-theory and infinite products}\label{sec:products}

\begin{ddd}\label{fiuewhfweiufhweiufhiu23r23r23r23r32r}
 Define the \emph{nonconnective algebraic $K$-theory functor} on homotopical Waldhausen categories by the composition
 \[ \wK := \bK \circ \wuloc \colon \Waldone \to \Sp \]
 of $\wuloc$ with the nonconnective algebraic $K$-theory functor $\bK \colon \mloc \to \Sp$ of \cite[Section~9]{BGT13}.
\end{ddd}

 Since $\bK$ is a colimit-preserving functor \cite[Theorem~9.8]{BGT13}, all structural results established about $\wuloc$ in \cref{sec:fundamental-theorems} carry over for $\wK$.
We refrain from stating them explicitly.

To finish our discussion of the general properties of algebraic $K$-theory  {of homotopical Waldhausen categories},
we address one of its more exotic properties, namely its compatibility with infinite products.
This was originally shown for connective algebraic $K$-theory by Carlsson in the setting of Waldhausen categories with a cylinder functor \cite{Carlsson95}.

\begin{theorem}\label{thm:products}
 Let $(\bC_i)_{i \in I}$ be a family of homotopical Waldhausen categories.
 Then the canonical map
 \[ \wK( \prod_{i \in I} \bC_i ) \xrightarrow{\sim} \prod_{i \in I} \wK(\bC_i) \]
 is an equivalence.
\end{theorem}

 {At the time of writing, we are not aware of an analogous statement being true for $\wuloc$.}
We derive \cref{thm:products} from the analogous statement for stable $\infty$-categories \cite[Theorem~1.3]{KWstable}.

\begin{lem}\label{lem:stab-and-prod}
 Let $(\cC_i)_{i \in I}$ be a family of right-exact $\infty$-categories.
 Then the canonical map
 \[ \uloc(\stab(\prod_{i \in I} \cC_i)) \to \uloc(\prod_{i \in I} \stab(\cC_i)) \]
 is an equivalence.
\end{lem}

\begin{rem}
 Note that the canonical functor $\stab(\prod_{i \in I} \cC_i) \to \prod_{i \in I} \stab(\cC_i)$ is not essentially surjective, and hence not an equivalence.
 The claim of \cref{lem:stab-and-prod} boils down to the assertion that any sequence $(\Sigma^{2n_i})_{i \in I}$ of iterated suspension functors induces the identity after applying $\cU_{loc}$. 
 The standard argument to show that $\Sigma^{2n}$ induces the identity map  after applying $\cU_{loc}$ requires an increasing number of applications of additivity as $n$ grows,
 and thus runs into problems for such an infinite sequence.
 
 However, there is a slight variation of the argument which only requires a fixed number of applications of additivity, regardless of $n$.
 For example, taking the coproduct of the cofiber sequences
 \[ \id \xrightarrow{\id} \id \to 0,\quad \Sigma \xrightarrow{0} \Sigma \to \Sigma \sqcup \Sigma^2,\quad \Sigma^2 \xrightarrow{\id} \Sigma^2 \to 0,\quad \Sigma^3 \xrightarrow{0} \Sigma^3 \to \Sigma^3 \sqcup \Sigma^4 \] 
 gives the cofiber sequence
 \[ \id \sqcup \Sigma \sqcup \Sigma^2 \sqcup \Sigma^3 \xrightarrow{\id \sqcup 0 \sqcup \id \sqcup 0} \id \sqcup \Sigma \sqcup \Sigma^2 \sqcup \Sigma^3 \to \Sigma \sqcup \Sigma^2 \sqcup \Sigma^3 \sqcup \Sigma^4.\]
 Similarly, we have a cofiber sequence
 \[ \id \sqcup \Sigma \sqcup \Sigma^2 \sqcup \Sigma^3 \xrightarrow{0 \sqcup \id \sqcup 0 \sqcup \id} \Sigma^4 \sqcup \Sigma \sqcup \Sigma^2 \sqcup \Sigma^3 \to \Sigma^4 \sqcup \Sigma \sqcup \Sigma^2 \sqcup \Sigma^3.\]
 Noting that the two cofiber sequences differ only in the second term, and that the second term is given by $\id \sqcup S$ and $\Sigma^4 \sqcup S$ for some  {endofunctor $S$}, respectively, it follows that $\bK(\id) \simeq \bK(\Sigma^4)$.
\end{rem}

\begin{proof}[Proof of \cref{lem:stab-and-prod}]
 Recall that we have an equivalence
 \[ \stab(\cC) \simeq \colim ( \cC \xrightarrow{\Sigma} \cC \xrightarrow{\Sigma} \cC \xrightarrow{\Sigma} \ldots ) \]
 for every right-exact $\infty$-category $\cC$.
 In particular, $\prod_{i \in I} \stab(\cC_i)$ admits the following description:
 Let $\bN$ denote the set of natural numbers, and let $\bN^I$ be the set of functions $I \to \bN$,
 equipped with the partial ordering such that $\alpha \leq \beta$ if and only if $\alpha(i) \leq \beta(i)$ for all $i \in I$.
 Then for the functor $\ccD \colon \bN^I \to \catrex$ which satisfies $\ccD(\alpha) \simeq \prod_{i \in I} \cC_i$ for all $\alpha$,
 and sends $\alpha \leq \beta$ to the functor
 \[ \prod_{i \in I} \Sigma^{\beta(i) - \alpha(i)} \colon \prod_{i \in I} \cC_i \to \prod_{i \in I} \cC_i, \]
 we have
 \[ \prod_{i \in I} \stab(\cC_i) \simeq \colim_{\bN^I} \ccD. \]
 Since $\prod_{i \in I} \stab(\cC_i) \simeq \stab(\prod_{i \in I} \stab(\cC_i))$, it suffices to consider the map
 \[ \uloc(\stab(\prod_{i \in I} \cC_i)) \to \uloc( \stab(\prod_{i \in I} \stab(\cC_i)) ). \]
 As $\stab$ commutes with filtered colimits in $\catrex$ and $\uloc$ commutes with filtered colimits in $\catex$, the canonical map
 \[ \colim_{\alpha \in \bN^I} \uloc(\stab(\ccD(\alpha))) \to \uloc(\stab(\prod_{i \in I} \stab(\cC_i)) ) \]
 is an equivalence, and the map we are interested in corresponds to the structural inclusion
 \[ \uloc(\stab(\ccD(0))) \to \colim_{\alpha \in \bN^I} \uloc(\stab(\ccD(\alpha))), \]
 where $0$ denotes the constant map $0 \colon I \to \bN$.
 
 To prove the claim, it is enough to show that all maps in the diagram $\uloc(\stab(\ccD))$ are equivalences.
 Let $(n_i)_{i \in I}$ be an arbitrary sequence of even natural numbers.
 Consider the endofunctor
 \[ S := \big( \bigoplus_{0 < k < n_i} \Sigma^k \big)_{i \in I} \colon \prod_{i \in I} \cC_i \to \prod_{i \in I} \cC_i. \]
 Then we have a cofiber sequence of exact functors $\prod_i \cC_i \to \prod_i \cC_i$
 \[ \id \oplus S \to \id \oplus S \to S \oplus ( \Sigma^{n_i} )_{i \in I}, \]
 in which the first transformation is given by
 \[ \big( \id \oplus \Sigma \oplus \Sigma^2 \oplus \ldots \oplus \Sigma^{n_i-1} \xrightarrow{\id \oplus 0 \oplus \id \oplus \ldots \oplus 0} \id \oplus \Sigma \oplus \Sigma^2 \oplus \ldots \oplus \Sigma^{n_i-1} \big)_{i \in I} \] ($\id$ and $0$ alternate).
 Moreover, there also exists a cofiber sequence
 \[ \id \oplus S \to ( \Sigma^{n_i} )_{i \in I} \oplus S \to S \oplus ( \Sigma^{n_i} )_{i \in I}, \]
 in which the first transformation is given by
 \[ \big( \id \oplus \Sigma \oplus \Sigma^2 \oplus \ldots \oplus \Sigma^{n_i-1} \xrightarrow{0 \oplus \id \oplus 0 \oplus \ldots \oplus \id} \Sigma^{n_i} \oplus \Sigma \oplus \Sigma^2 \oplus \ldots \oplus \Sigma^{n_i -1} \big)_{i \in I}. \]
 By virtue of the Additivity theorem, we have $\cU_{loc}(\id \oplus S) \simeq  \cU_{loc}((\Sigma^{n_i})_{i \in I} \oplus S)$, and conclude that
 \[ \cU_{loc}(\id) \simeq  \cU_{loc}((\Sigma^{n_i})_{i \in I}). \]
 Since the subset of all functions $I \to 2\bN$ is cofinal in $\bN^I$, this suffices to show that the diagram $\cU_{loc}(\stab(\ccD))$ is essentially constant,
 and thus proves our claim.
\end{proof}

\begin{proof}[Proof of \cref{thm:products}]
 Unravelling the definition of the functor $\wK$, we can factor the comparison map as
 \begin{align*}
  \bK( \cU_{loc}(\stab(\ell(\prod_{i \in I} \bC_i))))
  &\to \bK(\cU_{loc}(\stab(\prod_{i \in I} \ell(\bC_i)))) \\
  &\to \bK(\cU_{loc}(\prod_{i \in I}  \stab(\ell(\bC_i)))) \\
  &\to \prod_{i \in I} \bK(\cU_{loc}(\stab(\ell(\bC_i)))).
 \end{align*}
 The first map is an equivalence by \cite[Proposition~7.7.1]{CisBook},
 the second map is an equivalence by \cref{lem:stab-and-prod},
 and the third map is an equivalence by \cite[Theorem~1.3]{KWstable}.
\end{proof}


\section{Controlled retractive spaces over a bornological coarse space}\label{sec:controlled-retractive-spaces}

To produce a coarse variant of $A$-theory, we have to transfer the notion of controlled retractive spaces from \cite{Weiss02} and \cite{UW} to the setting of bornological coarse spaces.
While this is relatively straightforward, we try to make our treatment self-contained
(modulo the terminology introduced in \cite[Section~2]{BE} and \cite[Sections~2 and 3]{equicoarse},
which we will use freely throughout).
The main deviation from \cite{UW} in our treatment lies in the proof of the gluing lemma for controlled equivalences.

\subsection{Controlled CW-complexes}

Let $W$ be a $G$-space and let $K$ be a $G$-CW-complex relative to $W$.
Recall that a relative open $n$-cell of $K$ is a path component of $\sk_n(K) \setminus \sk_{n-1}(K)$, where $\sk_n(K)$ denotes the $n$-skeleton of $K$. A relative open cell of $K$ is a relative open $n$-cell for some $n$.

\begin{ddd}
 Denote by $\cells K$ the $G$-set of relative open cells of $K$.
 Let $\cells_k K$ denote the $G$-set of relative open $k$-cells in $K$, and set
 \[ \cells_{\leq k} K := \bigcup_{l \leq k} \cells_l K. \qedhere \]
\end{ddd}

\begin{ddd}
 For a subset $L \subseteq K$, denote by $\gen{L}$ the smallest (non-equivariant) subcomplex of $K$ containing $L$.
\end{ddd}

\begin{ddd}
 Let $X$ be a $G$-set.
 An \emph{$X$-labeling} of $K$ is a $G$-equivariant function
 \[ \lambda \colon \cells K \to X. \qedhere \]
\end{ddd}

Let $(X,\cU)$ be a $G$-coarse space.
Let $(K,\lambda_K)$ and $(L,\lambda_L)$ be $X$-labeled $G$-CW-complexes relative $W$, and let $\phi \colon K \to L$ be a $G$-equivariant and cellular map relative $W$.  
\begin{ddd}\label{def:controlled-map}
 The map $\phi$ is \emph{$(X,\cU)$-controlled} if there exists an entourage $U$ in $\cU$ such that
 \[ \{ (\lambda_L(e'),\lambda_K(e)) \mid e \in \cells K,\ e' \in \cells\gen{\phi(e)} \} \subseteq U. \]
 If $\phi$ is the identity map on $K$, we say that $(K,\lambda_K)$ is an \emph{$(X,\cU)$-controlled $G$-CW-complex}
 (or simply \emph{controlled $G$-CW-complex} if the $G$-coarse space $(X,\cU)$ is clear from context).
\end{ddd}

We denote by $\ccw{G}{X,\cU}{W}$ the category of $(X,\cU)$-controlled $G$-CW-complexes relative $W$,
and $(X,\cU)$-controlled, $G$-equivariant and cellular maps.

\begin{rem}\label{rem:comparison-UW-1}
 \cref{def:controlled-map} requires maps to be uniformly controlled by a single entourage,
 whereas \cite[Definition~2.3]{UW} enforces this condition only on each skeleton.
 See \cref{rem:comparison-UW-2} for further discussion.
\end{rem}

\begin{ddd}
 A \emph{subcomplex inclusion} is a morphism of the form
 \[ (K',\lambda|_{K'}) \hookrightarrow (K,\lambda) \]
 for some $G$-invariant subcomplex $K' \subseteq K$ relative $W$.
\end{ddd}

\begin{lem}\label{lem:pushouts}
 Consider a diagram $(L,\lambda_L) \xleftarrow{f} (K',\lambda|_{K'}) \xhookrightarrow{i} (K,\lambda)$ in which $i$ is a subcomplex inclusion.
 Then there exists a pushout diagram
 \[\xymatrix{
  (K',\lambda|_{K'}) \ar@{^(->}[r]^{i}\ar[d]_{f} & (K,\lambda)\ar[d]_{g} \\
  (L,\lambda_L)\ar@{^(_->}[r]^{j} & (L \sqcup_{K'} K, \lambda_{L \sqcup_{K'} K})
 }\]
 in $\ccw{G}{X,\cU}{W}$ such that $j$ is a subcomplex inclusion.
\end{lem}
\begin{proof}
 Let $L \sqcup_{K'} K$ be the usual pushout in $G$-CW-complexes. Note that $L \sqcup_{K'} K$ arises from $L$ by successively attaching cells from $K$ which do not lie in $K'$, so
 there is a canonical identification
 \[ \cells(L \sqcup_{K'} K) \cong \cells L \sqcup (\cells K \setminus \cells K'). \]
 We define a labeling on $L \sqcup_{K'} K$ by
 \[ \lambda_{L \sqcup_{K'} K} \colon \cells(L \sqcup_{K'} K) \to X,\quad e \mapsto \begin{cases}
 			\lambda_L(e) & e \in \cells L \\
 			\lambda_K(e) & e \in \cells K \setminus \cells K'
 		\end{cases}	
  \]
 The universal property is easy to verify.
\end{proof}

\begin{ddd}\label{def:tensor}
 Let $(K,\lambda)$ be an object in $\ccw{G}{X,\cU}{W}$, and let $L$ be a $G$-CW-complex.
 Define $(K,\lambda) \otimes L$ as the $G$-CW-complex given by the pushout
 \[\xymatrix{
  W \times L\ar[r]\ar[d] & K \times L\ar[d] \\
  W\ar[r] & K \otimes L
 }\]
 equipped with the labeling
 \[ \lambda \otimes L \colon \cells(K \otimes L) \cong \{ e \times e' \mid e \in \cells K,\ e' \in \cells L \} \to X,\ (e,e') \mapsto \lambda(e). \qedhere \]
\end{ddd}

The product $\otimes$ of \cref{def:tensor}  {defines} a functor
\[ \otimes \colon \ccw{G}{X,\cU}{W} \times G\CW \to \ccw{G}{X,\cU}{W}. \]
 {The following proposition summarizes some properties of the functor $\otimes$.}

\begin{prop}
 Let $(K',\lambda|_{K'}) \hookrightarrow (K,\lambda)$ be a subcomplex inclusion
 and let $L' \hookrightarrow L$ be an inclusion of $G$-CW-complexes.
 \begin{enumerate}
  \item The map $(K',\lambda|_{K'}) \otimes L \to (K,\lambda) \otimes L$ is a subcomplex inclusion.
  \item The map $(K,\lambda) \otimes L' \to (K,\lambda) \otimes L$ is a subcomplex inclusion.
  \item The pushout-product-axiom holds: The induced map
   \[(K,\lambda) \otimes L' \mathop{\sqcup}\limits_{(K',\lambda|_{K'}) \otimes L'} (K',\lambda|_{K'}) \otimes L \to (K,\lambda) \otimes L\]
   is a subcomplex inclusion.
 \end{enumerate}
\end{prop}

This follows from standard considerations about CW-complexes, and we omit the proof.

For any two objects $(K,\lambda_K)$ and $(L,\lambda_L)$ in $\ccw{G}{X,\cU}{W}$, the assigment
\[ [n] \mapsto \hom_{\ccw{G}{X,\cU}{W}}((K,\lambda_K) \otimes \abs{\Delta^n},(L,\lambda_L)) \]
gives a simplicial set $\hom((K,\lambda_K),(L,\lambda_L))_\bullet$, and  {induces} a simplicial enrichment of $\ccw{G}{X,\cU}{W}$ via the natural identification
$(K,\lambda_K) \otimes \Delta^0 \cong (K,\lambda_K)$.

\begin{lem}\label{lem:controlled-deformation-retraction}
 Let $(K',\lambda_K|_{K'}) \hookrightarrow (K,\lambda_K)$ be a subcomplex inclusion,
 and suppose that $L' \subseteq L$ is an inclusion of CW-complexes which is also a homotopy equivalence.
 Then
 \[ ((K',\lambda_K|_{K'}) \otimes L) \mathop{\sqcup}\limits_{(K',\lambda_K|_{K'}) \otimes L'} (K,\lambda_K) \otimes L' \hookrightarrow (K,\lambda_K) \otimes L \]
 is a controlled strong deformation retract.
\end{lem}
\begin{proof}
 Since $(S^{k-1} \times L) \mathop{\sqcup}\limits_{S^{k-1} \times L'} (D^k \times L')$
 is a strong deformation retract of $D^k \times L$ for all $k > 0$,
 one argues cell by cell to show that there exists a deformation retraction of $\sk_k(K,K') \otimes L$ onto $(\sk_{k-1}(K,K') \otimes L) \mathop{\sqcup}\limits_{\sk_{k-1}(K,K') \otimes L'} (\sk_k(K,K') \otimes L')$,
 where $\sk_k(K,K')$ denotes the relative $k$-skeleton of the pair $(K,K')$.
 {On each cell $e$, the trace of this deformation retraction is contained in the subcomplex generated by $e$.
 Since $K$ is a controlled $G$-CW-complex, this proves that the deformation retraction is a controlled homotopy.} 
 
 The desired deformation retraction on the whole complex $K$ is then obtained by stacking these homotopies.
 {It} is controlled since the control of $K$ is measured by a single entourage.
\end{proof}

\begin{kor}[Controlled homotopy extension property]\label{cor:chep}
 If $(K',\lambda_K|_{K'}) \hookrightarrow (K,\lambda_K)$ is a subcomplex inclusion,
 then any controlled map
 \[ ((K',\lambda_K|_{K'}) \otimes [0,1]) \mathop{\sqcup}\limits_{(K',\lambda_K|_{K'}) \otimes \{0\}} ((K,\lambda_K) \otimes \{0\}) \to (L,\lambda_L) \]
 extends to a controlled map
 \[ (K,\lambda_K) \otimes [0,1] \to (L,\lambda_L). \]
\end{kor}

\begin{kor}\label{lem:mapping-spaces-fibrations}
 If $\iota \colon (K',\lambda_K|_{K'}) \to (K,\lambda_K)$ is a subcomplex inclusion, the induced map
 \[ \iota^* \colon \hom((K,\lambda_K),(L,\lambda_L))_\bullet \to \hom((K',\lambda_K|_{K'}),(L,\lambda_L))_\bullet \]
 is a Kan fibration for any controlled $G$-CW-complex $(L,\lambda_L)$.
\end{kor}
\begin{proof}
 We have to show that the lifting problem
 \[\xymatrix{
  \Lambda^n_i\ar[r]\ar@{^(->}[d] & \hom((K,\lambda_K),(L,\lambda_L))_\bullet\ar[d]^{\iota^*} \\
  \Delta^n\ar[r]\ar@{-->}[ur] & \hom((K',\lambda_K|_{K'}),(L,\lambda_L))_\bullet
 }\]
 always has a solution.
 By definition, the given lifting problem corresponds to the extension problem
 \[\xymatrix{
  \big((K, \lambda_K) \otimes \abs{\Lambda^n_i}\big) \mathop{\sqcup}\limits_{(K',\lambda_K"_{K'}) \otimes \abs{\Lambda^n_i}} \big((K',\lambda_K|_{K'}) \otimes \abs{\Delta^n}\big)\ar[r]\ar@{^(->}[d] & (L,\lambda_L) \\
  (K,\lambda_K) \otimes \abs{\Delta^n}\ar@{-->}[ur] &
 }\]
 Since the inclusion $\abs{\Lambda^n_i} \subseteq \abs{\Delta^n}$ is a homotopy equivalence, the claim follows from \cref{lem:controlled-deformation-retraction}.
\end{proof}

\begin{kor}\label{cor:mapping-spaces-fibrant}
 $\ccw{G}{X,\cU}{W}$ is enriched in Kan complexes.
\end{kor}
\begin{proof}
 Note that $\ccw{G}{X,\cU}{W}$ has an initial object $(W,\emptyset)$, and that the unique map $(W,\emptyset) \to (K,\lambda_K)$ is a subcomplex inclusion.
 By \cref{lem:mapping-spaces-fibrations}, the restriction map
 \[ \hom((K,\lambda_K), (L,\lambda_L))_\bullet \to \hom((W,\emptyset), (L,\lambda_L))_\bullet \]
 is a Kan fibration for all $(L,\lambda_L)$.
 Since $(W,\emptyset) \otimes \Delta^n \cong (W,\emptyset)$ for all $n$, we have $\hom((W,\emptyset), (L,\lambda_L))_\bullet \cong \Delta^0$.
 Hence, $\hom((K,\lambda_K), (L,\lambda_L))_\bullet$ is a Kan complex.
\end{proof}

In the next step, we introduce a variant of the simplicial enrichment we have just discussed.

\begin{ddd}
 Let $Y \subseteq X$ be a $G$-invariant subset of $X$ and let $(K,\lambda)$ be a controlled $G$-CW-complex relative $W$.
 
 The \emph{restriction} $(K,\lambda)|_Y$ of $(K,\lambda)$ to $Y$ is defined as the subcomplex
 \[ (K,\lambda)|_Y := (K',\lambda|_{K'}), \]
 where $K'$ denotes the largest ($G$-invariant) subcomplex of $K$ satisfying $\lambda(\cells K') \subseteq Y$.
\end{ddd}

Let $\cY = \{ Y_i \}_{i \in I}$ be a big family of $G$-invariant subsets of $X$.
We introduce the simplicially enriched category $\ccw{G}{X,\cU}{W}^\cY$ of \emph{controlled $G$-CW-complexes mod $\cY$}:

It has the same objects as $\ccw{G}{X,\cU}{W}$, and its morphism spaces are given by
\[ \hom^\cY((K,\lambda_K),(L,\lambda_L))_\bullet := \colim_{i \in I} \hom((K,\lambda_K)|_{X \setminus Y_i},(L,\lambda_L))_\bullet, \]
where the colimit is taken along the obvious restriction maps.
The composition operation in this category is defined as follows.
Let $\phi \colon (K,\lambda_K)|_{X \setminus Y_i} \to (L,\lambda_L)$ and $\psi \colon (L,\lambda_L)|_{X \setminus Y_j} \to (M,\lambda_M)$ represent morphisms
$[\phi] \colon (K,\lambda_K) \to (L,\lambda_L)$ and $[\psi] \colon (L,\lambda_L) \to (M,\lambda_M)$ in $\ccw{G}{X,\cU}{W}^\cY$.
Since $\phi$ is a controlled morphism, there exists an entourage $U$ of $X$ such that $\lambda_L(\gen{\phi(e)}) \subseteq X \setminus Y_j$ if $\lambda_K(e) \in X \setminus U[Y_j]$.
As $\cY$ is a big family, there exists some $i' \geq i,j$ such that $U[Y_j] \subseteq Y_{i'}$.
Setting
\[ [\psi] \circ [\phi] := [ \psi \circ \phi|_{(K,\lambda_K)|_{X \setminus Y_{i'}}} ] \]
gives a well-defined composition operation and generalizes readily to higher simplices.

If $\cY$ is the trivial big family $\{ \emptyset \}$, the simplicially enriched category $\ccw{G}{X,\cU}{W}^\cY$ coincides with $\ccw{G}{X,\cU}{W}$.
The structure map
\[ \hom((K,\lambda_K),(L,\lambda_L))_\bullet \to \hom^\cY((K,\lambda_K),(L,\lambda_L))_\bullet \]
of the colimit provides a simplicially enriched functor
\[ q_\cY \colon \ccw{G}{X,\cU}{W} \to \ccw{G}{X,\cU}{W}^\cY. \]

\begin{ddd}
 A morphism $\phi \colon (K,\lambda_K) \to (L,\lambda_L)$ in $\ccw{G}{X,\cU}{W}$ is a \emph{controlled equivalence mod $\cY$} if the restriction map
 \[ q_\cY(\phi)^* \colon \hom^\cY((L,\lambda_L),(M,\lambda_M))_\bullet \to \hom^\cY((K,\lambda_K),(M,\lambda_M))_\bullet \]
 is a weak equivalence for all $(M,\lambda_M)$.
 
 For $\cY = \{ \emptyset \}$, we call $\phi$ simply a \emph{controlled equivalence}.
\end{ddd}

\begin{rem}\label{rem:equivalences-explicit}
 For any category $\bC$ enriched in Kan complexes and morphism $\phi \colon K \to L$ in $\bC$, the following are equivalent:
 \begin{enumerate}
  \item The restriction map $\phi^* \colon \bC(L,M)_\bullet \to \bC(K,M)_\bullet$ is a weak equivalence for all $M$ in $\bC$.
  \item There exists a morphism $\psi \colon L \to K$ and $1$-simplices $\eta_K \in \bC(K,K)_1$ and $\eta_L \in \bC(L,L)_1$ such that
   $d_0(\eta_K) = \id_K$, $d_1(\eta_K) = \psi\phi$, $d_0(\eta_L) = \id_L$ and $d_1(\eta_L) = \phi\psi$.
 \end{enumerate}
 Unwinding the latter characterization for $\ccw{G}{X,\cU}{W}^\cY$ gives an explicit description of controlled equivalences mod $\cY$
 in analogy to \cite[Definition~3.12]{UW}:
 A morphism $\phi \colon (K,\lambda_K) \to (L,\lambda_L)$ is a controlled equivalence mod $\cY$ if and only if there exist
 \begin{enumerate}
  \item some $i \in I$ and a morphism $\psi \colon (L,\lambda_K)|_{X \setminus Y_i} \to (K,\lambda_K)$;
  \item some $i' \geq i$ such that the composition $\psi \circ \phi|_{(K,\lambda_K)|_{X \setminus Y_{i'}}}$ is defined and some $i_K \geq i'$ together with\
   a morphism $(K,\lambda_K)|_{X \setminus Y_{i_K}} \otimes [0,1] \to (K,\lambda_K)$ restricting to the canonical subcomplex inclusion and $\psi \circ \phi|_{(K,\lambda_K)|_{X \setminus Y_{i'}}}$, respectively;
  \item some $i_L \geq i$ and a morphism $(L,\lambda_L)|_{X \setminus Y_{i_L}} \otimes [0,1] \to (L,\lambda_L)$ restricting to the canonical subcomplex inclusion and $\phi \circ \psi$, respectively. \qedhere
 \end{enumerate}
\end{rem}

\begin{ex}\label{ex:cofinal-inclusion}
 For every object $(K,\lambda)$ in $\ccw{G}{X}{W}$ and every $Y_i \in \cY$, the inclusion map $(K,\lambda)|_{X \setminus Y_i} \to (K,\lambda)$ is a controlled equivalence mod $\cY$.
\end{ex}
 
\begin{theorem}[Gluing lemma]\label{thm:gluing}
 Suppose we have a commutative diagram
 \[\label{eq:gluing-complexes}
 \xymatrix{
  (L_1,\lambda_{L_1})\ar[d]^-{\sim}_{\psi_L} & (K_1',\lambda_{K_1}|_{K_1'})\ar[l]_-{\phi_1}\ar@{^(->}[r]^-{\iota_1}\ar[d]^-{\sim}_{\psi_K|_{K_1'}} & (K_1,\lambda_{K_1})\ar[d]^-{\sim}_{\psi_K} \\
  (L_2,\lambda_{L_2}) & (K_2',\lambda_{K_2}|_{K_2'})\ar[l]_-{\phi_2}\ar@{^(->}[r]^-{\iota_2} & (K_2,\lambda_{K_2}) \\
 }\]
 in which all vertical morphisms are controlled equivalences mod $\cY$, and in which $\iota_1$ and $\iota_2$ are subcomplex inclusions.
 
 Then the induced morphism
 \[ (L_1,\lambda_{L_1}) \mathop{\sqcup}\limits_{(K_1',\lambda_{K_1}|_{K_1'})} (K_1,\lambda_{K_1}) \to (L_2,\lambda_{L_2}) \mathop{\sqcup}\limits_{(K_2',\lambda_{K_2}|_{K_2'})} (K_2,\lambda_{K_2}) \]
 is a controlled equivalence mod $\cY$.
\end{theorem}
\begin{proof}
 To increase legibility, we suppress the labelings from notation.
 
 Let $r \in \{1,2\}$ and let $Y_i \in \cY$.
 Denote by $\phi^*_r(L_r|_{X \setminus Y_i})$ the largest subcomplex of $K_r'$ satisfying
 \[ \phi_r(\phi^*_r(L_r|_{X \setminus Y_i})) \subseteq L_r|_{X \setminus Y_i}. \]
 Moreover, denote by $K_r^i$ the largest subcomplex of $K_r $ such that the following two conditions are satisfied:
 \begin{enumerate}
 \item $K_r^i \cap K_r' = \phi^{*}_{r}(L_r|_{X \setminus Y_i})$;
 \item for every cell $e$ in $K_{r}\setminus K^{i}_{r}$ we have $\lambda_{K_{r}}(e)\in X\setminus Y_{i}$.
 \end{enumerate}
 Then we have a canonical isomorphism
 \[ L_r|_{X \setminus Y_i} \mathop{\sqcup}\limits_{\phi^*_r(L_r|_{X \setminus Y_i})} K_r^i \cong ( L_r \mathop{\sqcup}\limits_{K_r'} K_r)|_{X \setminus Y_i}. \] 
 Since directed colimits commute with finite limits in simplicial sets, we obtain a pullback square
 \[\xymatrix{
  \colim\limits_{i \in I} \hom(( L_r \mathop{\sqcup}\limits_{K_r'} K_r)|_{X \setminus Y_i}, M)_\bullet\ar[r]\ar[d] & \colim\limits_{i \in I} \hom(K_r^i, M)_\bullet\ar[d] \\
  \colim\limits_{i \in I} \hom(L_r|_{X \setminus Y_i}, M)_\bullet\ar[r] & \colim\limits_{i \in I} \hom(\phi^*_r(L_r|_{X \setminus Y_i}), M)_\bullet
 }\]
 Suppose that $\phi_r$ is $U$-controlled.
 Let $i \in I$ and pick $j \geq i$ such that $U[Y_i] \subseteq Y_j$.
 Then $K_r'|_{X \setminus Y_j} \subseteq \phi_r^*(L_r|_{X \setminus Y_i})$
 and $\phi_r^*(L_r|_{X \setminus Y_j}) \subseteq K_r|_{X \setminus Y_i}$.
 In particular, we obtain isomorphisms
 \[ \colim\limits_{i \in I} \hom(\phi^*_r(L_r|_{X \setminus Y_i}), M)_\bullet \cong \colim\limits_{i \in I} \hom(K_r'|_{X \setminus Y_i},M)_\bullet \]
 and
 \[ \colim\limits_{i \in I} \hom(K_r^i, M)_\bullet \cong \colim\limits_{i \in I} \hom(K_r|_{X \setminus Y_i},M)_\bullet. \]
 It follows that the commutative square
 \[\xymatrix{
  \hom^\cY(L_r \mathop{\sqcup}\limits_{K_r'} K_r,M)_\bullet \ar[r]\ar[d] & \hom^\cY(K_r,M)_\bullet\ar[d]^{\iota_r^*} \\
  \hom^\cY(L_r,M)_\bullet \ar[r] & \hom^\cY(K_r',M)_\bullet
 }\]
 is a pullback.
 All corners of this square are Kan complexes by \cref{cor:mapping-spaces-fibrant},
 and the restriction map $\iota_r^*$ is a Kan fibration by \cref{lem:mapping-spaces-fibrations}.
 Hence, the square is a homotopy pullback.
 
 Since the transformation of homotopy pullback squares induced by $\psi_K$ and $\psi_L$ is an equivalence on all but the top left corner by assumption,
 it follows that the induced map on the top left corner is also an equivalence.
 This proves the claim of the theorem.
\end{proof}

 \subsection{Controlled retractive spaces}

Let $W$ be a $G$-space and let $(X,\cU)$ be a $G$-coarse space.

\begin{ddd}
 An \emph{$(X,\cU)$-controlled retractive space} $(K,\lambda,r)$ over $W$ is an $(X,\cU)$-controlled $G$-CW-complex $(K,\lambda)$ over $W$
 together with a $G$-equivariant retraction $r \colon K \to W$ to the structural inclusion $W \to K$.
 
 A \emph{morphism} of $(X,\cU)$-controlled retractive spaces is an $(X,\cU)$-controlled morphism of $(X,\cU$)-controlled $G$-CW-complexes
 which is additionally compatible with the chosen retractions.
\end{ddd}

The $(X,\cU)$-controlled retractive spaces and their morphisms form a category $\ret{G}{X,\cU}{W}{}$.

\begin{ddd}
 Define the \emph{cofibrations} $\cof\ret{G}{X,\cU}{W}{} \subseteq \ret{G}{X,\cU}{W}{}$ to be the collection of all morphisms
 which are isomorphic to a morphism given by a subcomplex inclusion.
\end{ddd}

\begin{ddd}
 Let $\cY$ be a big family of $G$-invariant subsets of $X$.
 A \emph{weak equivalence mod $\cY$} is a morphism
 which is sent to a controlled equivalence mod $\cY$ by the canonical functor $\ret{G}{X,\cU}{W}{} \to \ccw{G}{X,\cU}{W}$.
 Denote by $h^\cY\ret{G}{X,\cU}{W}{}$ the collection of all weak equivalences mod $\cY$.
 
 If $\cY = \{ \emptyset \}$, we typically omit $\cY$ from notation.
\end{ddd}

\begin{prop}\label{prop:ret-waldhausen-cat}
 The triple $(\ret{G}{X,\cU}{W}{}, \cof\ret{G}{X,\cU}{W}{}, h^\cY\ret{G}{X,\cU}{W}{})$ is a homotopical Waldhausen category.
\end{prop}
\begin{proof}
 The category $\ret{G}{X,\cU}{W}{}$ has a zero object given by the controlled $G$-CW-complex $(W,\emptyset)$ together with the retraction $\id_W$.

 The unique map $(W,\emptyset,\id_W) \to (K,\lambda,r)$ is a cofibration for every controlled retractive space,
 and all isomorphisms are cofibrations.
 The existence of pushouts along cofibrations follows from \cref{lem:pushouts} together with the observation
 that the pushout of retractive spaces inherits a retraction by the universal property of the pushout.
 This also implies that $\cof\ret{G}{X,\cU}{W}{}$ is a subcategory.
 
 By definition, all isomorphisms are $h^\cY$-equivalences.
 The gluing lemma for $h^\cY$-equivalences is precisely \cref{thm:gluing}
 
 Since the weak equivalences are pulled back from a fibrant simplicially enriched category,
 they are closed under retracts and satisfy the two-out-of-six property.
 Moreover, we can use the product construction from \cref{def:tensor} to show that all morphisms in $\ret{G}{X,\cU}{W}{}$
 admit a factorization into a cofibration followed by a weak equivalence:
 If $\phi \colon (K,\lambda_K,r_K) \to (L,\lambda_L,r_L)$ is a morphism, define the mapping cylinder of $\phi$ by the pushout
 \[\xymatrix{
  (K,\lambda_K,r_K) \otimes \{1\} \ar[r]^-{\phi}\ar@{^(->}[d] & (L,\lambda_L,r_L)\ar@{^(->}[d] \\
  (K,\lambda_K,r_K) \otimes [0,1]\ar[r] & M(\phi)
 }\]
 Since the left vertical subcomplex inclusion is a weak equivalence,
 \cref{thm:gluing} implies that the right vertical morphism is also a weak equivalence.
 Hence, the map $\pi \colon M(\phi) \to (L,\lambda_L,r_L)$ induced by 
 $(K,\lambda_K,r_K) \otimes [0,1] \to (K,\lambda_K,r_K) \xrightarrow{\phi} (L,\lambda_L,r_L)$ and $\id_L$
 via the universal property of the pushout is a weak equivalence.
 The cofibration $(K,\lambda_K,r_K) \otimes \{0\} \hookrightarrow (K,\lambda_K,r_K) \otimes [0,1]$ induces a cofibration $(K,\lambda_K,r_K) \hookrightarrow M(\phi)$,
 and it is easy to check that the composition $(K,\lambda_K,r_K) \hookrightarrow M(\phi) \xrightarrow{\pi} (L,\lambda_L,r_L)$ equals $\phi$.
\end{proof}

We will usually abbreviate notation and write $(\ret{G}{X,\cU}{W}{},h^\cY)$ for the Waldhausen category $(\ret{G}{X,\cU}{W}{}, \cof\ret{G}{X,\cU}{W}{}, h^\cY\ret{G}{X,\cU}{W}{})$.

Note that $\ret{G}{X,\cU}{W}{}$ admits arbitrary coproducts.
Since we wish to take the algebraic $K$-theory of the Waldhausen category $(\ret{G}{X,\cU}{W}{},h^\cY)$,
we impose additional finiteness properties on objects in $\ret{G}{X,\cU}{W}{}$.
In order to do so, we have to additionally assume that $X$ comes equipped with a bornology.

\begin{ddd}\label{def:locally-finite}
 Let $(X,\cB,\cU)$ be a $G$-bornological coarse space.
 An object $(K,\lambda)$ of $\ccw{G}{X,\cU}{W}$ is called \emph{locally finite} if $\lambda^{-1}(B)$ is a finite set for every bounded subset $B$ of $X$.
 
 A controlled retractive space is \emph{locally finite} if it is locally finite as an object in $\ccw{G}{X,\cU}{W}$.
 Denote the full subcategory of locally finite controlled retractive spaces by $\ret{G}{X,\cB,\cU}{W}{\lf}$.
\end{ddd}

To save space, we will typically suppress the bornology and coarse structure on $X$ from now on and write $\ret{G}{X}{W}{\lf}$ for $\ret{G}{X,\cB,\cU}{W}{\lf}$.
Since $\ret{G}{X}{W}{\lf}$ is closed under pushouts along cofibrations,
it forms a full homotopical Waldhausen subcategory $(\ret{G}{X}{W}{\lf},h^\cY)$ of $(\ret{G}{X}{W}{}, h^\cY)$.

\begin{rem}\label{rem:comparison-UW-2}
 The finiteness condition introduced in \cref{def:locally-finite} is weaker than the notion of finiteness employed in \cite[Definition~3.3]{UW}:
 \begin{enumerate}
  \item We make no requirements about the image of the retraction map.
   This condition is irrelevant for the question whether the algebraic $K$-theory of $\ret{G}{X}{W}{\lf}$ defines a coarse homology theory (as a functor of $X$),
   and only affects how this theory behaves as a functor with respect to $W$.
   One can impose conditions on the images of the retraction maps without affecting the discussion in \cref{sec:a-homology} and \cref{sec:injectivity} except \cref{prop:a-homology-strongly-additive}.
   For example, one could require the images of the retraction maps to be contained in a $G$-compact subset of $W$ in order to obtain a functor in $W$ which is compatible with directed colimits.
  \item We do not require complexes to be finite-dimensional.
   This is important in our context as it produces a strongly additive coarse homology theory, see \cref{prop:a-homology-strongly-additive}.
   In \cite{UW}, this additional requirement ensures that locally finite
   complexes are uniformly controlled by a single entourage.
   Moreover, finite-dimensionality plays an important role in the proof of the Farrell--Jones conjecture for finitely $\cV\cC yc$-amenable groups in \cite{ELPUW18}.
   As illustrated by \cite[Sections~4 and 5]{UW}, the proof that the algebraic $K$-theory of $\ret{G}{X}{W}{\lf}$ produces a continuous coarse homology theory
   also goes through if we impose finite-dimensionality as an additional requirement (but we see no reason why the resulting theory should still be strongly additive). \qedhere
 \end{enumerate}
\end{rem}

To conclude this section, we discuss the functoriality of $\ret{G}{X}{W}{}$ and $\ret{G}{X}{W}{\lf}$ in $X$.
Let $f \colon (X,\cU_X) \to (Y,\cU_Y)$ be a morphism of $G$-coarse spaces.
Since $f$ is controlled, we obtain an induced exact functor
\[ f_* \colon (\ret{G}{X}{W}{},h) \to (\ret{G}{Y}{W}{},h),\quad (K,\lambda,r) \mapsto (K, f\circ \lambda, r). \]
Evidently, this defines a functor
\[ (\ret{G}{-}{W}{}, h) \colon G\mathbf{Coarse} \to \Waldone. \]
If $X$ and $Y$ are $G$-bornological coarse spaces and $f$ is in addition proper (i.e.~a morphism a $G$-bornological coarse spaces),
then the exact functor $f_*$ restricts to an exact functor
\[ f_* \colon (\ret{G}{X}{W}{\lf},h) \to (\ret{G}{Y}{W}{\lf},h). \]
Consequently, we have a functor
\[ (\ret{G}{-}{W}{\lf},h) \colon G\BC \to \Waldone. \]


\section{Coarse \texorpdfstring{$A$}{A}-homology}\label{sec:a-homology}

The goal of this section is to prove that the functor
\[ \ret{G}{-}{W}{\lf} \colon G\BC \to \Waldone \]
induces an equivariant coarse homology theory
after composing with the algebraic $K$-theory functor $\wK$.
For the convenience of the reader, we recall the definition of an equivariant coarse homology theory from \cite[Definition~3.10]{equicoarse}.

\begin{ddd}
 Let $\cC$ be a cocomplete stable $\infty$-category.
 A functor
 \[ E \colon G\BC \to \cC \]
 is a \emph{$\cC$-valued $G$-equivariant coarse homology theory} if it satisfies the following properties:
 \begin{enumerate}
  \item $E$ is \emph{coarsely invariant}: For every $G$-bornological coarse space $X$, the morphism $E(\{0,1\}_{max,max} \otimes X \to X)$ is an equivalence.
  \item $E$ is \emph{coarsely excisive}: We have $E(\emptyset) \simeq 0$, and for every complementary equivariant pair $(Z,\cY)$ in $X$ the induced square
   \[\xymatrix{
    E(Z \cap \cY)\ar[r]\ar[d] & E(Z)\ar[d] \\
    E(\cY)\ar[r] & E(X)
   }\]
   is a pushout, where $E(\cY):=\colim_{Y\in \cY} E(Y)$.
  \item $E$ \emph{vanishes on flasques}: If $X$ is flasque, then $E(X) \simeq 0$.
  \item $E$ is \emph{$u$-continuous}: For every $G$-bornological coarse space $X$, the natural map
  \[ \colim\limits_{U \in \cU^G} E(X_U) \to E(X) \]
  is an equivalence, where $X_{U}$ denotes the $G$-bornological coarse space obtained from $X$ by replacing the coarse structure by the smallest coarse structure containing $U$. \qedhere
 \end{enumerate}
\end{ddd}

\begin{ddd}[{\cite[Definition~5.15]{equicoarse}}]
 An equivariant coarse homology theory $E$ is \emph{continuous} if the following holds:
 For every filtered family $\cY = \{ Y_i \}_{i \in I}$ of $G$-invariant subsets such that for every $G$-invariant, locally finite subset $F$ of $X$
 there exists some $i \in I$ with $F \subseteq Y_i$, the canonical map $E(\cY) \to E(X)$ is an equivalence.
\end{ddd}

Recall \cref{efgieofewfefewfwefwefewfwef} of the functor $\wuloc$ and
\cref{fiuewhfweiufhweiufhiu23r23r23r23r32r} of $\wK$.
\begin{ddd}
 \emph{Universal coarse $A$-homology} (relative to the base space $W$) is the functor
 \[ \UAX^G_W \colon G\BC \xrightarrow{\ret{G}{-}{W}{\lf}} \Waldone \xrightarrow{\wuloc} \mloc. \]
 \emph{Coarse equivariant $A$-homology} (relative to the base space $W$) is the functor
 \[ \AX^G_W := \bK \circ \UAX^G_W \simeq \wK \circ \ret{G}{-}{W}{\lf} \colon G\BC \to \Sp. \qedhere \]
\end{ddd}

\begin{theorem}\label{thm:a-homology}
 Both $\UAX^G_W$ and $\AX^G_W$ are continuous $G$-equivariant coarse homology theories.
\end{theorem}

The proof of \cref{thm:a-homology} occpuies the rest of this section.

\begin{lem}\label{lem:close-morphisms}
 If $f$ and $f'$ are close morphisms of $G$-bornological coarse spaces, they induce naturally isomorphic functors
 \[ f_* \cong f'_* \colon  \ret{G}{X}{W}{} \to \ret{G}{Y}{W}{}. \]
\end{lem}
\begin{proof}
 Since $f$ and $f'$ are close, the identity morphism on underlying spaces defines a natural isomorphism $f_*(K,\lambda,r) \xrightarrow{\cong} f'_*(K,\lambda,r)$.
\end{proof}

\begin{kor}\label{cor:a-homology-coarsely-invariant}
 The functor $\UAX^G_W$ is coarsely invariant.
\end{kor}
\begin{proof}
 This follows directly from \cref{lem:close-morphisms} since isomorphic functors induce equivalent maps
 and morphisms $\{0,1\}_{max,max} \otimes X \to Y$ correspond to pairs of close maps $X \to Y$.
\end{proof}

\begin{lem}\label{lem:a-homology-u-continuous}
 The functor $\UAX^G_W$ is $u$-continuous.
\end{lem}
\begin{proof}
 Since each controlled retractive space and each controlled morphism is controlled by a single entourage,
 the category of controlled retractive spaces can be written as a  filtered colimit
 \[ \ret{G}{X}{W}{\lf} \simeq \colim_{U \in \cU^G_X} \ret{G}{X_U}{W}{\lf} \]
 in $\Waldone$. The claim now follows from \cref{prop:k-filtered-colimits}.
\end{proof}

\begin{prop}\label{prop:flasque-swindle}
 Let $X$ be a $G$-bornological coarse space.
 Suppose there exists a sequence $(s_n \colon X \to X)_{n \in \bbN}$ of morphisms such that the following holds:
 \begin{enumerate}
  \item\label{it:flasque-swindle-1} $s_0 = \id_X$;
  \item\label{it:flasque-swindle-2}  {$\bigcup_n (s_n \times s_{n+1})(\diag(X))$ is an entourage of $X$, where $\diag(X)$ denotes the diagonal of $X$;}
  \item\label{it:flasque-swindle-3} $\bigcup_n (s_n \times s_n)(U)$ is an entourage of $X$ for every entourage $U$;
  \item\label{it:flasque-swindle-4} for every bounded subset $B$ of $X$ exists some natural number $n_0$ with $B \cap s_n(X) = \emptyset$   for all $n \geq n_0$.
 \end{enumerate}
 Then $\ret{G}{X}{W}{\lf}$ admits an Eilenberg swindle.
\end{prop}
\begin{proof}
 Each morphism $s_n$ induces an endofunctor $(s_n)_*$.
 By assumption~\eqref{it:flasque-swindle-3}, we obtain the exact endofunctor
 \[ S := \bigvee_{n \in \bbN} (s_n)_* \colon \ret{G}{X}{W}{} \to \ret{G}{X}{W}{}. \]
 We claim that $S$ preserves locally finite objects:
 Suppose $(K,\lambda)$ is locally finite, and write $S(K,\lambda) = (K_\infty, \lambda_\infty)$.
 Let $B$ be a bounded subset of $X$.
 Assumption~\eqref{it:flasque-swindle-4} implies that there exists some natural number $n_0$ such that
 \[ \lambda_\infty^{-1}(B) = \bigcup_{n \in \bbN} \lambda^{-1} \circ s_n^{-1}(B) = \bigcup_{n < n_0} \lambda^{-1} \circ s_n^{-1}(B). \]
 Hence, $S(K,\lambda)$ is locally finite, and $S$ restricts to an endofunctor
 \[ S \colon \ret{G}{X}{W}{\lf} \to \ret{G}{X}{W}{\lf}. \]
 Using assumptions~\eqref{it:flasque-swindle-1} and \eqref{it:flasque-swindle-2} together with \cref{lem:close-morphisms},
 it follows that there exists a natural isomorphism $\id_{\ret{G}{X}{W}{\lf}} \vee S \cong S$,
 so $S$ is an Eilenberg swindle on $\ret{G}{X}{W}{\lf}$.
\end{proof}

\begin{rem}\label{rem:flasque-shift}
 If $f \colon X \to X$ is a morphism which implements flasqueness of $X$ in the sense of \cite[Definition~3.8]{equicoarse},
 then the sequence $(f^n)_{n \geq 0}$ satisfies the assumptions of \cref{prop:flasque-swindle}.
 {If $G$ is the trivial group, the assumption of  \cref{prop:flasque-swindle} is equivalent to the notion of flasqueness in the generalized sense as defined in \cite[Definition~3.27]{BE} (see also \cite[Definition~3.10]{Wright}).}
\end{rem}

 \begin{kor}\label{cor:a-homology-flasques}
 The functor $\UAX^G_W$ vanishes on flasque $G$-bornological coarse spaces.
\end{kor}
\begin{proof}
 If $f \colon X \to X$ implements flasqueness of $X$, we can apply \cref{prop:flasque-swindle} by \cref{rem:flasque-shift}.
 Then the Additivity theorem (\cref{thm:additivity}) implies $\UAX^G_W(X) \simeq 0$.
\end{proof}

\begin{ddd}\label{wefjiwefwfwefwef}
 Let $X$ be a $G$-bornological coarse space and let $\cY=(Y_{i})_{i\in I}$ be a big family of $G$-invariant subsets of $X$.
 Then we define
 \[ \ret{G}{\cY}{W}{\lf} := \colim_{i\in I}\ret{G}{Y_{i}}{W}{\lf}.\qedhere \]
\end{ddd}

\begin{rem}
 In the situation of \cref{wefjiwefwfwefwef} we can and will identify $\ret{G}{\cY}{W}{\lf}$ with the full subcategory of $\ret{G}{X}{W}{\lf}$ spanned by those objects $(K,\lambda,r)$ for which there exists some $i\in I$ with $\lambda(\cells K) \subseteq Y_i$.
\end{rem}

Since $\ret{G}{\cY}{W}{\lf}$ is closed under pushouts, it forms a full homotopical Waldhausen subcategory of $\ret{G}{X}{W}{\lf}$.

\begin{const}\label{const:zigzags}
 Let $X$ be a $G$-bornological coarse space and consider a zig-zag
 \[ (K,\lambda_K,r_K) \xrightarrow{\phi} (L,\lambda_L,r_L) \xleftarrow{\alpha} (L',\lambda_{L'},r_{L'}) \]
 in $\ret{G}{X}{W}{}$ such that $\alpha$ is a controlled equivalence.
 
 Forgetting the retractions, we find by \cref{rem:equivalences-explicit} a controlled map of relative $G$-CW-complexes
 $\beta \colon (L,\lambda_L) \to (L',\lambda_{L'})$ as well as controlled homotopies $\id_L \simeq \alpha\beta$ and $\id_{L'} \simeq \beta\alpha$ relative $W$.
 Let
 \[ (K,\lambda_K) \xrightarrow{\iota} (M,\lambda_M) \xrightarrow{\pi} (L',\lambda_{L'}) \]
 be the factorization of $\beta\phi$ obtained by applying the mapping cylinder construction from the proof of \cref{prop:ret-waldhausen-cat}.
 Since there exists a controlled homotopy $\alpha\beta\phi \simeq \phi$,
 the controlled maps $\phi$ and $\alpha$ on the ends of the mapping cylinder extend to a controlled map $\psi \colon (M,\lambda_M) \to (L,\lambda_L)$.
 Then $\psi$ becomes a morphism of retractive spaces by equipping $M$ with the retraction $r_M := r_L\psi$.
 Since $\psi$ restricts to $\phi$ on $K$, the inclusion $\iota$ is then also a morphism of retractive spaces.
 In particular, we have a factorization
 \[ \phi \colon (K,\lambda_K,r_K) \xrightarrow{\iota} (M,\lambda_M,r_M) \xrightarrow[\sim]{\psi} (L,\lambda_L, r_L) \]
 in $\ret{G}{X}{W}{}$.
 If $(K,\lambda_K,r_K)$ and $(L',\lambda_{L'},r_{L'})$  are objects of a full homotopical Waldhausen subcategory $\bR$ of $\ret{G}{X}{W}{}$ which is closed under tensoring with $[0,1]$,
 then $(M,\lambda_M,r_M)$ also belongs to $\bR$.
\end{const}

Recall the notion of domination from \cref{ergiofrggergrge}. Let furthermore $\ret{G}{X}{W}{\lf}^{h^{\cY}}$ denote the full subcategory of $\ret{G}{X}{W}{\lf}$ on the objects which are equivalent to zero mod $\cY$.

\begin{prop}\label{prop:excision-fibers}
 Let $X$ be a $G$-bornological coarse space and let $\cY$ be a big family of $G$-invariant subsets of $X$.
 Then every object in $\ret{G}{X}{W}{\lf}^{h^{\cY}}$ is dominated by an object in $\ret{G}{\cY}{W}{\lf}$:
 
 For every object $(K,\lambda_K,r_K) \in  \ret{G}{X}{W}{\lf}^{h^\cY}$ there exist objects
 \[ (K',\lambda_{K'},r_{K'}) \in \ret{G}{X}{W}{\lf}^{h^\cY}, \quad (L,\lambda_L,r_L) \in \ret{G}{\cY}{W}{\lf} \]
 and morphisms
 \[ \phi \colon (K',\lambda_{K'},r_{K'}) \to (L,\lambda_L,r_L), \quad \psi \colon (L,\lambda_L,r_L) \to (K,\lambda_K,r_K) \]
 such that the composition $\psi\phi$ is a controlled equivalence.
\end{prop}
\begin{proof}
 Let $(K,\lambda,r)$ be controlled contractible mod $\cY$.
 By \cref{rem:equivalences-explicit}, this means that there exists a controlled homotopy $\eta \colon (K,\lambda)|_{X \setminus Y_i} \otimes [0,1] \to (K,\lambda)$ between the canonical inclusion and the zero map $(K',\lambda)|_{X \setminus Y_i} \xrightarrow{r} W \hookrightarrow K$
 for some $i \in I$.
 By the controlled homotopy extension property~\ref{cor:chep}, the map
 \[ ((K,\lambda)|_{X \setminus Y_i} \otimes [0,1]) \mathop{\sqcup}\limits_{(K,\lambda)|_{X \setminus Y_i} \otimes \{0\}} (K,\lambda) \otimes \{0\}) \xrightarrow{\eta \sqcup \id_K} (K,\lambda) \]
 extends to a controlled homotopy $\overline{\eta} \colon (K,\lambda) \otimes [0,1] \to (K,\lambda)$.
 Let $U$ be an entourage such that $\overline{\eta}$ is $U$-controlled.
 Since $\cY$ is a big family, there exists some $j \in I$ such that $U[Y_i] \subseteq Y_j$.
 The endpoint $\overline{\eta}_1$ of the homotopy $\overline{\eta}$ differs from the retraction map only on cells labelled by points in $Y_i$.
 Hence, $\overline{\eta}_1$ factors as
 \[\xymatrix{
  (K,\lambda)\ar[r]^{\pi} & (K,\lambda)|_{Y_j}\ar@{^(->}[r]^-{\iota} & (K,\lambda).
 }\]
 This gives rise to a domination
 \[  (K,\lambda, \iota\pi r) \xrightarrow{\pi} (K,\lambda,r)|_{Y_j} \xrightarrow{\iota} (K,\lambda,r), \]
 whose existence is exactly what we needed to show.
\end{proof}

\begin{prop}\label{prop:excision-cofibers}
 Let $X$ be a $G$-bornological coarse space, and let $(Z,\cY)$ be an equivariant complementary pair.
 Then the exact functor
 \[ (\ret{G}{Z}{W}{\lf}, {h^{Z \cap \cY}}) \to (\ret{G}{X}{W}{\lf}, h^\cY) \]
 satisfies the assumptions of \cref{thm:approximation}
\end{prop}
\begin{proof}
 Using \cref{rem:equivalences-explicit}, it is easy to see that the first part of the approximation property holds.
 
 To verify the second part of the approximation property, let $\phi \colon (K,\lambda_K) \to (L,\lambda_L)$
 be an arbitrary morphism in $\ret{G}{X}{W}{\lf}$ such that $(K,\lambda_K)$ lies in $\ret{G}{Z}{W}{\lf}$ (which is considered here, for simplicity,  as a full subcategory of $\ret{G}{X}{W}{\lf}$).
 Choose an entourage $U$ such that $\phi$ is $U$-controlled.
 Let $i \in I$ be such that $Z \cup Y_i = X$, and let $j \geq i$ sich that $U^{-1}[Y_i] \subseteq Y_j$.
 Then $U[Z \setminus (Z \cap Y_j)] \subseteq X \setminus Y_i$, so $\phi$ restricts to a morphism
 \[ \phi' \colon (K,\lambda_K)|_{Z \setminus (Z \cap Y_j)} \to (L,\lambda_L)|_{X \setminus Y_i}. \]
 Since $X \setminus Y_i \subseteq Z$, the morphism $\phi'$ lies in $\ret{G}{Z}{W}{\lf}$.
 In particular, we may form the pushout
 \[\xymatrix{
  (K,\lambda_K)|_{Z \setminus (Z \cap Y_j)}\ar[r]^{\phi'}\ar@{>->}[d] & (L,\lambda_L)|_{X \setminus Y_i}\ar@{>->}[d] \\
  (K,\lambda_K)\ar[r] & (K,\lambda_K) \mathop{\sqcup}\limits_{(K,\lambda_K)|_{Z \setminus (Z \cap Y_j)}} (L,\lambda_L)|_{X \setminus Y_i}
 }\]
 in $\ret{G}{Z}{W}{\lf}$.
 By \cref{thm:gluing} (Glueing Lemma) and \cref{ex:cofinal-inclusion},
 the right vertical morphism is a controlled equivalence mod $Z \cap \cY$.
 By the universal property of the pushout, we obtain a factorization of the canonical inclusion $(L,\lambda_L)|_{X \setminus Y_i} \rightarrowtail (L,\lambda_L)$ as
 \[ (L,\lambda_L)|_{X \setminus Y_i} \rightarrowtail (K,\lambda_K) \mathop{\sqcup}\limits_{(K,\lambda_K)|_{Z \setminus (Z \cap Y_j)}} (L,\lambda_L)|_{X \setminus Y_i} \xrightarrow{\phi''} (L,\lambda_L). \]
 By the two-out-of-three property of weak equivalences, it follows that $\phi''$ is a controlled equivalence mod $\cY$.
 Moreover, we have a factorization of $\phi$ as
 \[ (K,\lambda_K) \to (K,\lambda_K) \mathop{\sqcup}\limits_{(K,\lambda_K)|_{Z \setminus (Z \cap Y_j)}} (L,\lambda_L)|_{X \setminus Y_i} \xrightarrow[\sim]{\phi''} (L,\lambda_L), \]
 which verifies the second part of the approximation property.
\end{proof}

\begin{kor}\label{cor:a-homology-coarsely-excisive}
 The functor $\UAX^G_W$ is coarsely excisive.
\end{kor}
\begin{proof}
 Let $X$ be a $G$-bornological coarse space, and let $(Z,\cY)$ be an equivariant complementary pair.
 Consider the commutative diagram
 \[\xymatrix{
  (\ret{G}{Z}{W}{\lf}^{h^{Z \cap \cY}}, h)\ar[r]\ar[d] & (\ret{G}{Z}{W}{\lf}, h)\ar[r]\ar[d] & (\ret{G}{Z}{W}{\lf}, h^{Z \cap \cY})\ar[d] \\
  (\ret{G}{X}{W}{\lf}^{h^{\cY}}, h)\ar[r] & (\ret{G}{X}{W}{\lf}, h)\ar[r] & (\ret{G}{X}{W}{\lf}, h^{\cY})
 }\]
 By \cref{thm:fibration}, both rows induce cofiber sequences in $\mloc$ upon application of $\wuloc$.
 Moreover, the induced map of cofibers is an equivalence by \cref{prop:excision-cofibers} and \cref{thm:approximation}.
 Therefore, we have a pushout square
 \[\xymatrix{
  \wuloc(\ret{G}{Z}{W}{\lf}^{h^{Z \cap \cY}}, h)\ar[r]\ar[d] & \wuloc(\ret{G}{Z}{W}{\lf}, h)\ar[d] \\
  \wuloc(\ret{G}{X}{W}{\lf}^{h^{\cY}}, h)\ar[r] & \wuloc(\ret{G}{X}{W}{\lf}, h)
 }\]
 in $\mloc$.
 \cref{const:zigzags} shows that both inclusion functors
 \[ (\ret{G}{Z \cap \cY}{W}{\lf}, h) \to (\ret{G}{Z}{W}{\lf}^{h^{Z \cap \cY}}, h) \text{ and } (\ret{G}{\cY}{W}{\lf}, h) \to (\ret{G}{X}{W}{\lf}^{h^{\cY}}, h) \]
 admit a mapping cylinder argument.
 Therefore, \cref{prop:excision-fibers} and \cref{thm:cofinality} imply that these functors
 induce equivalences upon application of $\wuloc$, which yields the desired pushout square.
\end{proof}

\begin{proof}[Proof of \cref{thm:a-homology}]
 The functor $\UAX^G_W$ is an equivariant coarse homology theory by \cref{cor:a-homology-coarsely-invariant},
 \cref{lem:a-homology-u-continuous}, \cref{cor:a-homology-flasques} and \cref{cor:a-homology-coarsely-excisive}.
 
 Let $\cY = \{ Y_i \}_{i \in I}$ be a filtered family of $G$-invariant subsets such that for every $G$-invariant, locally finite subset $F$ of $X$
 there exists some $i \in I$ with $F \subseteq Y_i$.
 Since for every locally finite retractive space $(K,\lambda)$ the image $\lambda(\cells K)$ of the labeling is a locally finite subset of $X$,
 it follows that $\ret{G}{X}{W}{\lf}$ is the filtered union
 \[ \ret{G}{X}{W}{\lf} = \bigcup_{i \in I} \ret{G}{Y_i}{W}{\lf}. \]
 Continuity follows from \cref{prop:k-filtered-colimits}.
 
 Since $\bK$ commutes with filtered colimits, it is immediate that $\AX^G_W$ is also a continuous equivariant coarse homology theory.
\end{proof}


\section{Split injectivity of the \texorpdfstring{$A$}{A}-theoretic assembly map}\label{sec:injectivity}

The goal of this section is to use the equivariant coarse homology theory $\AX^G_W$ constructed in \cref{sec:a-homology}
to derive the split injectivity results stated in the introduction.

To avoid a barrage of definitions at the beginning,
the proof of the main result (\cref{thm:a-cp}) is split into a sequence of individual statements,
each of which addresses another additional property of coarse $A$-homology.

\subsection{The relation to classical \texorpdfstring{$A$}{A}-theory}\label{sec:classical-a}
First of all, let us relate coarse $A$-homology to the classical $A$-theory functor due to Waldhausen \cite[Section~2]{Wald85}.

Let $Q$ be a topological space.
The objects of the category of retractive spaces $\retuc{}{Q}{}$ are CW-complexes $K$ relative $Q$ equipped with a retraction $r \colon K \to Q$ to the inclusion $Q \hookrightarrow K$.
Morphisms in this category are cellular maps over and under $Q$.
Cofibrations are those morphisms which are isomorphic to subcomplex inclusions,
and weak equivalences are those morphisms which, as maps relative $Q$, are homotopy equivalences.
The full homotopical Waldhausen subcategory $\retuc{}{Q}{\f}$ of finite retractive spaces over $Q$ is spanned by those retractive spaces
which arise from $Q$ by attaching only a finite number of cells.

The ordinary (nonconnective) $A$-theory functor is given by the composition
\[ \bA \colon \Top \xrightarrow{\retuc{}{-}{\f}} \Waldone \xrightarrow{\wK} \Sp, \]
where $\Top$ denotes the ordinary $1$-category of topological spaces and continuous maps.

\begin{rem}\label{rem:ret-functor}In this remark we explain the functoriality of $\retuc{}{Q}{\f}$ with respect to the space $Q$.
	If $f \colon Q \to Q'$ is a continous map, the induced functor $\retuc{}{f}{\f} \colon \retuc{}{Q}{\f} \to \retuc{}{Q}{\f}$ sends the object   $(K,r)$ to the object $(f_{*}K,f_{*}r)$, where $f_{*}K$ is defined by the pushout
	\[\xymatrix{
	 Q\ar[r]\ar[d]_{f} & K\ar[d] \\
	 Q'\ar[r] & f_*K
	}\ ,\] and the 
	   retraction $f_*r \colon f_*K \to Q'$ is determined by the  universal property of the pushout. 
	  The latter is also  
	   used to define $\retuc{}{f}{\f}$ on morphisms $(K,r)\to (K^{\prime},r^{\prime})$ in $\retuc{}{Q}{\f}$.
	
	Note that the construction $K\mapsto f_{*}K$ can be made strictly functorial by suitable choices on the point-set level:
	Since we require that $Q$ is contained in $K$ (as the $(-1)$-skeleton of $K$), the pushout $f_*K$ can be chosen to have the underlying set $Q' \cup (K \setminus Q)$, equipped with the appropriate topology and induced filtration.
\end{rem}

Let now $G$ be a discrete group and let $P$ be the total space of a principal $G$-bundle.
 {Then $P$ gives rise to the functor
\[ \retuc{}{P \times_G -}{\f} \colon \orb(G) \to \Waldone,\quad S \mapsto \retuc{}{P \times_G S}{\f} \]
and thus to the $\orb(G)$-spectrum
\[ \bA_P := \wK \circ \retuc{}{P \times_G -}{\f}   \colon \orb(G) \to \Sp. \]}

Let $G_{can,min}$ be the $G$-bornological coarse space whose  $G$-set is  $G$,
 equipped with the minimal bornology and the coarse structure generated by the  subset $\{G(F \times F) \mid F \subseteq G \text{ finite} \}$ of the   power set of $G\times G$.
\begin{prop}\label{prop:org-spectra-cat}
 There is a zig-zag of equivalences
 \[ \ret{G}{G_{can,min} \otimes (-)_{min,max}}{P}{\lf} \simeq \retuc{}{P \times_G -}{\f} \]
 between functors $\orb(G) \to \Waldone$.
\end{prop}

The proof of \cref{prop:org-spectra-cat} relies on the following construction which we will reuse later on.

\begin{ddd}\label{def:extraordinary-point}
 Let $I$ be a set. Define the topological space $I^+$ to have underlying set $I \sqcup \{+\}$,
 and define a non-empty subset of $I^+$ to be closed if and only if it contains the distinguished point $+$.
\end{ddd}

Let $X$ be a $G$-bornological coarse space and consider an object $(K,\lambda,r)$ of $\ret{G}{X}{P}{\lf}$.
 {Let $\pi_{0}(X)$ denote the $G$-set of coarse components of $X$.}
Due to the control conditions imposed on $(K,\lambda)$, we observe that for every coarse component $X_0$ of $X$
the set $\lambda^{-1}(X_0)$ is the set of relative open cells of a subcomplex of $K$.
Hence, $\lambda$ induces a continuous map
\[ \lambda_0 \colon K \to \pi_0(X)^+, \]
(which sends $P$ to $\{+\}$)
and we have a canonical identification (the wedge sum indicates the coproduct of CW-complexes relative $P$)
\[ K \cong \bigvee_{ X_{0}\in \pi_0(X)} \lambda_0^{-1}(\{ X_{0},+\}). \] 

\begin{proof}[Proof of \cref{prop:org-spectra-cat}]
 {Let us abbreviate
 \[ F_c := \ret{G}{G_{can,min} \otimes (-)_{min,max}}{P}{\lf} \colon \orb(G) \to \Waldone \]
 and
 \[ F_u := \retuc{}{P \times_G -}{\f} \colon \orb(G) \to \Waldone. \]
We furthermore define the functor $F_c' \colon \orb(G) \to \Waldone$ by
 \[ S \mapsto \ret{G}{G_{can,min}}{P \times S}{\lf}, \]
 where we use that $\ret{G}{G_{can,min}}{-}{\lf}$ is a functor on topological spaces with a $G$-action in the same way as explained in \cref{rem:ret-functor}.}

 Let $S \in \orb(G)$, and let $(K,r,\lambda)$ be an object in $F_c'(S)$. The composed map $K \xrightarrow{r} P \times S \xrightarrow{p} S$ induces a $G$-equivariant function $\lambda_S \colon \cells K \to S$. By taking the pushout
 \[\xymatrix{
 	P \times S\ar[r]\ar[d]_{p} & K\ar[d] \\
 	P\ar[r] & p_*K
 }\]
 of $K$ along the projection map, we obtain a retractive space over $P$. Define a labeling $p_*\lambda$ on $p_*K$ by setting
 \[ p_*\lambda \colon \cells p_*K \cong \cells K \to G \times S,\quad e \mapsto (\lambda(e),\lambda_S(e)). \]
 It is easy to check that $(p_*K,p_*r,p_*\lambda)$ is a controlled complex over $G_{can,min} \otimes S_{min,max}$.
 Since $(K,r,\lambda)$ is locally finite over $G_{can,min}$, there exist only finitely many $G$-cells in $K$. So $(p_*K,p_*r,p_*\lambda)$ is also locally finite over $G_{can,min} \otimes S_{min,max}$.
 It is straightforward to check that this construction extends to a functor $F_c'(S) \to F_c(S)$.

 Conversely, let $(K,r,\lambda)$ be an object in $F_c(S)$.
 Since $\pi_0(G_{can,min} \otimes S_{min,max}) \cong S$,
 we have seen that the complex $K$ canonically decomposes as a coproduct (relative $P$)
 \[ K \cong \bigvee_{s \in S} K_s. \]
 Set $\widetilde{K} := \bigsqcup_{s \in S} K_s$.
 Since each summand $K_s$ is a retractive space over $P$, $\widetilde{K}$ is canonically a retractive space over $P \times S$.
 Moreover, there is an induced labeling
 \[ \widetilde{\lambda} \colon \cells \widetilde{K} \cong \cells K \xrightarrow{\lambda} G \times S \to G. \]
 Since $\{1\} \times S$ is bounded in $G_{can,min} \otimes S_{min,max}$, the complex $K$ contains only finitely many $G$-cells.
 Hence, $\widetilde{K}$ together with the labeling $\widetilde{\lambda}$ is an object in $F_c'(S)$.
 It is again straightforward to check that this construction extends to a functor $F_c(S) \to F_c'(S)$.
 
 Moreover, the functors $F_c(S) \to F_c'(S)$ and $F_c'(S) \to F_c(S)$ are easily seen to be inverse to each other.
 One checks that the collection of functors $\{ F_c'(S) \to F_c(S) \}_{S \in \orb(G)}$ defines a natural transformation $F_c' \to F_c$,
 so that we have a natural equivalence $F_c' \xrightarrow{\sim} F_c$.
 
 Recall from \cite[Section~2.1]{Wald85} the equivariant version of the category of retractive spaces:
 The category of finite $G$-retractive spaces $\retuc{G}{P}{\f}$ has as objects (free) $G$-CW-complexes $K$ relative $P$ equipped with an equivariant retraction $r \colon K \to P$ such that $K$ arises from $P$ by attaching finitely many free $G$-cells.
 Morphisms in this category are $G$-equivariant and cellular maps over and under $P$.
 Cofibrations are morphisms isomorphic to subcomplex inclusions,
 and weak equivalences are those morphisms which, as maps relative $P$, are (equivariant) homotopy  equivalences.
 Similar to the explanation in \cref{rem:ret-functor}, we obtain a functor $F_u' := \retuc{G}{P \times -}{\f}$ from $\orb(G)$ to $\Waldone$.
 
 Since the underlying CW-complex of any object in $F_c'(S)$ contains only finitely many $G$-cells, forgetting labelings defines a natural transformation $F_c' \to F_u'$.
 In order to construct an inverse we consider an
    object $(K,r)$ of $F_u'(S)$. Since $G$ acts freely on $P$ all $G$-cells of $K$ are free.
    We then choose a base point in every $G$ orbit on $\cells K$ and define the equivariant labeling $\lambda:\cells K\to G$ such that its sends the chosen  base points to $1\in G$. Note that $(K,\lambda,r)$ belongs to $F_{c}^{\prime}(S)$ since $\cells K/G$ is finite.
          We define the inverse equivalence such that it sends $(K,r)$ to $(K,\lambda,r)$. 
    
 Finally, taking quotients by the $G$-actions induces a natural equivalence $F_u' \to F_u$, see \cite[Lemma~2.1.3]{Wald85}.
 {In sum, we obtain the desired zig-zag of natural equivalences
 \[ F_c \xleftarrow{\sim} F_c' \xrightarrow{\sim} F_u' \xrightarrow{\sim} F_u. \qedhere \]}
\end{proof}

By applying $\wK$, \cref{prop:org-spectra-cat} has the following immediate consequence.

\begin{kor}\label{prop:org-spectra}
 There is an equivalence of $\orb(G)$-spectra
 \[ \AX^G_P(G_{can,min} \otimes (-)_{min,max}) \simeq \bA_P(-). \]
\end{kor}

\subsection{Transfers}
In the next step, we show that coarse $A$-homology admits a certain amount of contravariant functoriality.
To describe this additional functoriality, we use the $\infty$-category $G\BC_{\tr}$ introduced in \cite[Section~2.2]{coarsetrans},
which comes equipped with a functor $\iota \colon G\BC \to G\BC_{\tr}$.

\begin{ddd}[{\cite[Definition~2.53]{coarsetrans}}]\label{def:transfers}
 An equivariant coarse homology theory $E \colon G\BC \to \cC$ \emph{admits transfers}
 if there exists a functor
 \[ E_{\tr} \colon G\BC_{\tr} \to \cC \]
 such that $E \simeq E_{\tr} \circ \iota$.
\end{ddd}

 For our purposes, the precise definition of $G\BC_{\tr}$ is not relevant. Instead, we will rely on an explicit criterion to describe functors from $G\BC_{\tr}$ to the nerve of a strict $2$-category.

\begin{ddd}[{\cite[Definition~2.14]{coarsetrans}}]
 Let $W$ and $X$ be $G$-bornological coarse spaces.
 A \emph{bounded covering} $w \colon W \to X$ is a map of $G$-sets satisfying the following properties:
 \begin{enumerate}
  \item $w$ is controlled and bornological.
  \item The coarse structure on $W$ coincides with the coarse structure
   \[ \{ U \cap \bigcup_{W_{0} \in \pi_0(W)} (W_{0} \times W_{0}) \mid U \in w^{-1}\cU_X \}. \]
  \item For every bounded subset $B$ of $W$, there exists a pairwise coarsely disjoint partition $B = \bigcup_\alpha B_\alpha$ such that $w|_{B_\alpha} \colon [B_\alpha] \to [w(B_\alpha)]$ is an isomorphism of coarse spaces, where $[-]$ denotes the coarse closure. \qedhere
 \end{enumerate}
\end{ddd}

\begin{ddd}[{\cite[Definition~2.19]{coarsetrans}}]
 Let $W$, $U$, $X$ and $Y$ be $G$-bornological coarse spaces.
 An \emph{admissible square}
 \[\xymatrix{
  W\ar[r]^-{f}\ar[d]_-{w} & U\ar[d]^-{u} \\
  X\ar[r]^-{g} & Y
 }\]
 is a pullback square in $G\mathbf{Coarse}$ in which both $f$ and $g$ are proper and bornological\footnote{A bornological map sends bounded subsets to bounded subsets. In contrast, for a proper map preimages of bounded subsets are bounded.}, and $u$ is a bounded covering.
\end{ddd}

Let $\cC$ be a $(2,1)$-category.
Then $\cC$ gives rise to an $\infty$-category by applying the ordinary nerve to each morphism category, yielding a fibrant simplicial category $N_*\cC$, and then taking the homotopy coherent nerve to obtain
\[ \fN(\cC) := \hcnerve(N_*\cC). \]
According to \cite[Lemma~3.1]{coarsetrans}, the following data determine a functor from $G\BC_{\tr} \to \fN(\cC)$:
\begin{enumerate}
 \item a functor $F \colon G\BC \to \cC$, where we consider $\cC$ as a $1$-category by forgetting all $2$-morphisms;
 \item for every bounded covering $w \colon W \to X$ a morphism $w^* \colon F(X) \to F(W)$;
 \item for every two composable bounded coverings $v \colon V \to W$ and $w \colon W \to X$ a $2$-morphism $a_{v,w} \colon (w \circ v)^* \Rightarrow v^* \circ w^*$;
 \item for every admissible square
  \[\xymatrix{
   V\ar[r]^{d}\ar[d]_{v} & W\ar[d]^{w} \\
   X\ar[r]^{f} & Y
  }\]
 a $2$-morphism $b_{f,w} \colon d_* \circ v^* \Rightarrow w^* \circ d_*$, where we write $d_*$ and $f_*$ for $F(d)$ and $F(f)$, respectively.
\end{enumerate}
Moreover, these data have to satisfy the following compatibility conditions:
\begin{enumerate}
 \item if a bounded covering $w \colon W \to X$ is an isomorphism on the underlying $G$-coarse spaces, then $w^* = w^{-1}_*$ (note that $w^{-1}$ is a morphism of $G$-bornological coarse spaces);
 \item if two composable bounded coverings $v \colon V \to w$ and $w \colon W \to X$ are both isomorphisms on the underlying $G$-coarse spaces,
  then $a_{v,w}$ is the identity of $(wv)^{-1}_* = v^{-1}_* \circ w^{-1}_*$;
 \item for every three composable bounded coverings $u \colon U \to V$, $v \colon V \to W$ and $w \colon W \to X$, we have
  \[ (u^* \circ a_{v,w})a_{u,wv} = (a_{u,v} \circ w^*)a_{vu,w}; \]
 \item for every admissible square
  \[\xymatrix{
   V\ar[r]^{d}\ar[d]_{v} & W\ar[d]^{w} \\
   X\ar[r]^{f} & Y
  }\]
  in which $w$ (and hence also $v$) is an isomorphism of $G$-coarse spaces, $b_{f,w}$ is the identity of $d_* \circ v^{-1}_* = w^{-1}_* \circ f_*$;
 \item for every admissible square
  \[\xymatrix{
   V\ar[r]^{d}\ar[d]_{v} & W\ar[d]^{w} \\
   X\ar[r]^{f} & Y
  }\]
 in which $d$ and $f$ are identity morphisms, $b_{f,w}$ is the identity of $v^* = w^*$;
 \item for every diagram
  \[\xymatrix{
   U\ar[r]^{d}\ar[d]^{u} & V\ar[r]^{e}\ar[d]^{v} & W\ar[d]^{w} \\
   X\ar[r]^{f} & Y\ar[r]^{g} & Z
  }\]
  in which both squares are admissible we have
  \[ b_{gf,w} = (b_{g,w} \circ f_*)(e_* \circ b_{f,v});\]
 \item for every diagram
  \[\xymatrix{
   T\ar[r]^{f}\ar[d]_{t} & U\ar[d]^{u} \\
   V\ar[r]^{g}\ar[d]_{v} & W\ar[d]^{w} \\
   X\ar[r]^{h} & Y
  }\]
  in which both squares are admissible we have
  \[ (a_{w,u} \circ h_*)b_{h,wu} = (u^* \circ b_{h,w})(b_{g,v} \circ v^*)(f_* \circ a_{t,v}).\]
\end{enumerate}
Extend the category $\Waldone$ to a $(2,1)$-category $\Wald$ whose underlying $1$-category is $\Waldone$,
but whose $2$-morphisms are natural isomorphisms of functors.
Our goal is to use the above criterion to extend $\ret{G}{-}{P}{\lf}$ to a functor $G\BC_{\tr} \to \fN(\Wald)$.

Let $w \colon W \to X$ be a bounded covering.
The idea for the construction of the transfer is rather simple:
As we have seen, any object $(K,\lambda,r)$ in $\ret{G}{X}{P}{\lf}$ decomposes as a coproduct indexed over the coarse components of $X$.
Since the restriction of $w$ to a coarse  
component $W_{0}$ is an isomorphism of coarse spaces onto the image $w(W_{0})$ which sends bounded subsets to bounded subsets,
the object $w^*(K,\lambda,r)$ is obtained by taking for each coarse component $W_{0}$ of $W$
a copy of the summand of $K$ indexed by $w(W_{0})$ and equipping it with the obvious labelling induced by $\lambda$. 
One easily checks that the local finiteness condition is preserved.
In order to verify that this transfer construction is sufficiently coherent to define a functor $G\BC_{\tr} \to \fN(\Wald)$, we express the transfer in terms of universal constructions.

Let $(K,\lambda,r)$ be an object in $\ret{G}{X}{P}{\lf}$.
We may form the pullback
\begin{equation}\label{eq:transfer}
 \xymatrix{
  w^*K\ar[r]\ar[d] & K\ar[d]^-{\lambda_0} \\
  \pi_0(W)^+\ar[r]^-{\pi_0(w)^+} & \pi_0(X)^+
 }
\end{equation}
It is straightforward to check that the tensor product from \cref{def:tensor} can be used to construct $w^*K$ explicitly via the formula
\[ \bigvee_{X_{0} \in \pi_0(X)} \lambda_0^{-1}(  \{X_{0},+ \}) \otimes \pi_0(w)^{-1}(X_{0}). \]
In particular, $w^*K$ is a $G$-CW-complex relative $P$ and inherits the retraction
\[ w^*r \colon w^*K \to K \to P. \]
Note that the set of relative open cells $\cells(w^*K)$ fits into a pullback diagram
\[\xymatrix{
 \cells(w^*K)\ar[r]^-{p}\ar[d]^-{q} & \cells K \ar[d]^{\lambda} \\
 \pi_0(W)\ar[r]^-{\pi_0(w)} & \pi_0(X)
}\]
of $G$-sets.
Since the restriction of $w$ to each coarse component of $W$ is an isomorphism of coarse spaces,
it makes sense to define the labeling
\[ w^*\lambda \colon \cells(w^*K) \to W,\quad e \mapsto  w|_{q(e)}^{-1}(\lambda(p(e))). \]
Note that the map $w^*K \to \pi_0(W)^+$ in \eqref{eq:transfer} is given by $(w^*\lambda)_0$.
Define
\[ w^*(K,\lambda,r) := (w^*K,w^*\lambda,w^*r). \]
By the universal property of the pullback, we obtain an exact functor
\[ w^* \colon \ret{G}{X}{P}{\lf} \to \ret{G}{W}{P}{\lf}. \]
If $w$ is an isomorphism on the underlying $G$-coarse spaces, we can choose $w^*K := K$,
we let the top horizontal map in \eqref{eq:transfer} be the identity on $K$,
and we let the left vertical map in \eqref{eq:transfer} be given by $\pi_0(w^{-1})^+ \circ \lambda_0$.
It follows that $w^* = w^{-1}_*$ in this case.

Let now $v \colon V \to W$ and $w \colon W \to X$ be two bounded coverings.
Then the universal property of the pullback produces for every $(K,\lambda,r)$ in $\ret{G}{X}{P}{\lf}$ a natural isomorphism
\[ a_{v,w} \colon (w \circ v)^*(K,\lambda,r) \to (v^* \circ w^*)(K,\lambda,r) \]
which can be chosen to be the identity transformation if both $v$ and $w$ are isomorphisms on the underlying $G$-coarse spaces.
Note that, again by the universal property of the pullback, the relation
\[ (u^* \circ a_{v,w})a_{u,wv} = (a_{u,v} \circ w^*)a_{vu,w} \]
holds for every three composable bounded coverings $u \colon U \to V$, $v \colon V \to W$ and $w \colon W \to X$.

Suppose that
\[\xymatrix{
 V\ar[r]^{d}\ar[d]_{v} & W\ar[d]^{w} \\
 X\ar[r]^{f} & Y
}\]
is an admissible square.

\begin{lem}\label{lem:admissible-components}
 The induced square
 \[\xymatrix{
  \pi_0(V)\ar[r]^{\pi_0(d)}\ar[d]_{\pi_0(v)} & \pi_0(W)\ar[d]^{\pi_0(w)} \\
  \pi_0(X)\ar[r]^{\pi_0(f)} & \pi_0(Y)
 }\]
 is a pullback of $G$-sets.
\end{lem}
\begin{proof}
 Consider the induced function $c \colon \pi_{0}(V) \to \pi_0(X) \times_{\pi_0(Y)} \pi_0(W)$.
 
 Let $(X_0,W_0) \in \pi_0(X) \times_{\pi_0(Y)} \pi_0(W)$ and let $x \in X_0$.
 Since $w$ is a bounded covering, there exists some $\widetilde y \in W_0$ such that $w(\widetilde y) = f(x)$.
 As $V$ is the pullback $X \times_Y W$, this shows that there exists a point $\widetilde x \in V$ such that $v(\widetilde x) = x$.
 Let $V_0$ be the coarse component of $\widetilde x$.
 Then $v(V_0) = X_0$ and $\pi_0(d)(V_0) = W_0$, so $c$ is surjective.
 
 Let $V_0, V_1 \in \pi_0(V)$ such that $c(V_0) = c(V_1)$.
 Pick $x \in \pi_0(v)(V_0) = \pi_0(v)(V_1)$.
 Since $v$ is a bounded covering, there exist uniquely determined points $\widetilde x_0 \in V_0$ and $\widetilde x_1 \in V_1$ such that $v(\widetilde x_0) = x = v(\widetilde x_1)$.
 Then $d(\widetilde x_0)$ and $d(\widetilde x_1)$ lie in the same coarse component of $W$ and map to the same point under $w$.
 Since $w$ is a bounded covering, it follows that $d(\widetilde x_0) = d(\widetilde x_1)$, and therefore $\widetilde x_0 = \widetilde x_1$ because $V \cong X \times_Y W$.
 Hence, $c$ is injective.
\end{proof}

Using \cref{lem:admissible-components}, both squares in the commutative diagram
\[\xymatrix{
 v^*K\ar[r]\ar[d] & \pi_0(V)^+\ar[d]_-{\pi_0(v)^+}\ar[r]^-{\pi_0(d)^+} & \pi_0(W)^+\ar[d]^-{\pi_0(w)^+} \\
 K\ar[r]^-{\lambda_0} & \pi_0(X)^+\ar[r]^-{\pi_0(f)^+} & \pi_0(Y)^+
}\]
are pullbacks.
Observing that the composition $v^*K \to \pi_0(V)^+ \xrightarrow{\pi_0(d)^+} \pi_0(W)^+$ equals $(d \circ v^*\lambda)_0$ and that $\pi_0(f)^+ \circ \lambda_0 = (f \circ \lambda)_0$,
we see that the large outer pullback square defines $w^*(f_*K)$.
Hence, the universal property of the pullback yields a natural isomorphism
\[ b_{f,w} \colon (d_* \circ v^*)(K,\lambda,r) \to (w^* \circ f_*)(K,\lambda,r). \]
If $v$ and $w$ are isomorphisms on the underlying $G$-coarse spaces, we can arrange that $b_{f,w}$ is the identity transformation.
Also, if $d$ and $f$ are identity morphisms, we may pick $b_{f,w}$ to be the identity transformation.

Let now
\[\xymatrix{
 U\ar[r]^{d}\ar[d]^{u} & V\ar[r]^{e}\ar[d]^{v} & W\ar[d]^{w} \\
 X\ar[r]^{f} & Y\ar[r]^{g} & Z
}\]
be a commutative diagram in which both squares are admissible.
By \cite[Lemma~2.20 and~2.21]{coarsetrans}, the outer square is also admissible.
For every object $(K,\lambda,r)$ in $\ret{G}{X}{P}{\lf}$, we obtain a commutative diagram
\[\xymatrix{
 v^*K\ar[r]^-{(v^*\lambda)_0}\ar[d] & \pi_0(U)^+\ar[d]^-{\pi_0(u)^+}\ar[r]^-{\pi_0(d)^+} & \pi_0(V)^+\ar[d]^-{\pi_0(v)^+}\ar[r]^-{\pi_0(e)^+} & \pi_0(W)^+\ar[d]^-{\pi_0(w)^+} \\
 K\ar[r]^-{\lambda_0} & \pi_0(X)^+\ar[r]^-{\pi_0(f)^+} & \pi_0(Y)^+\ar[r]^{\pi_0(g)^+} & \pi_0(Z)^+
}\]
Since $\pi_0(f)^+ \circ \lambda_0 = (f \circ \lambda)_0$, $\pi_0(g)^+ \circ \pi_0(f)^+ \circ \lambda_0 = (g \circ f \circ \lambda)_0$,
$\pi_0(d)^+ \circ (v^*\lambda)_0 = (d \circ v^*\lambda)_0$ and $\pi_0(e)^+ \circ \pi_0(d)^+ \circ (v^*\lambda)_0 = (e \circ d \circ v^*\lambda)_0$,
we observe that the transformations $b_{gf,w}$ and $(b_{g,w} \circ f_*)(e_* \circ b_{f,v})$ are both induced by the identity maps on each of the entries in the above diagram.
Consequently, they must be equal by the universal property of the pullback.

Similarly, the two possible natural isomorphisms
\[ (u^* \circ w^* \circ h_*)(K,\lambda,r) \xrightarrow{\cong} (f_* \circ t^* \circ v^*)(K,\lambda,r) \]
arise from the different ways to compose the squares in the diagram
\[\xymatrix{
 t^*v^*K\ar[r]\ar[d] & \pi_0(T)^+\ar[r]^-{\pi_0(f)^+}\ar[d]^-{\pi_0(t)^+} & \pi_0(U)^+\ar[d]^-{\pi_0(u)^+} \\
 v^*K\ar[r]\ar[d] & \pi_0(V)^+\ar[r]^-{\pi_0(g)^+}\ar[d]^-{\pi_0(v)^+} & \pi_0(W)^+\ar[d]^-{\pi_0(w)^+} \\
 K\ar[r]^{\lambda_0} & \pi_0(X)^+\ar[r]^-{\pi_0(h)^+} & \pi_0(Y)^+
}\]
Again, since all squares are pullbacks and both transformations are induced by the identity maps on all components, they must coincide.
The next proposition summarizes the discussion up to this point.

\begin{prop}\label{prop:ret-with-transfers}
 The above data determine a functor
 \[ \ret{G}{-}{P}{\lf}_{\tr} \colon G\BC_{\tr} \to \fN(\Wald) \]
 such that the diagram
 \[\xymatrix@C=4em{
  N(G\BC\ar[r]^-{\ret{G}{-}{P}{\lf}})\ar[d]^-{\iota} & N(\Waldone)\ar[d] \\
  G\BC_{\tr}\ar[r]^-{\ret{G}{-}{P}{\lf}_{tr}} & \fN(\Wald)
 }\]
 commutes.
\end{prop}
\begin{proof}
 Apply \cite[Lemma~3.1]{coarsetrans}.
\end{proof}

Let $\RelCat$ denote the $(2,1)$-category of relative categories, functors of relative categories and natural isomorphisms.
The forgetful functor $u \colon \Waldone \to \RelCatone$ from \cref{sec:right-exact} extends
to a functor $u_2 \colon \Wald \to \RelCat$.
Composing $\fN(u_2)$ with $\ret{G}{-}{P}{\lf}_{tr}$, we obtain a functor
\[ \fN(u_2) \circ \ret{G}{-}{P}{\lf}_{\tr} \colon G\BC_{\tr} \to \fN(\RelCat). \]
As described in \cite[Definition~A.23 and proof of Proposition~A.25]{GHN17}, the functor
\[ L \colon \RelCatone \to \sSet^+, \quad (\bC,\bW) \mapsto (N\bC, N\bW) \]
extends to a map
\[ L_2 \colon \fN(\RelCat) \to \hcnerve(\sSet^+), \]
and thus induces a map
\[  \ret{G}{-}{P}{\lf}_\ell' \colon G\BC_{tr} \to \hcnerve(\sSet^+) \]
of (large) simplicial sets.
Composing as in \cref{sec:right-exact} with a fibrant replacement functor $R \colon \sSet^+ \to \sSet^+$ yields a functor
\[ \ret{G}{-}{P}{\lf}_\ell \colon G\BC_{\tr} \to \hcnerve((\sSet^+)^\cf) = \catinf. \]
\cref{prop:ret-with-transfers} implies that we have
\begin{equation}\label{eq:transfers-compatible}
 \ret{G}{-}{P}{\lf}_\ell \circ \iota \simeq \ell \circ \ret{G}{-}{P}{\lf}.
\end{equation}
We can now prove that the $\orb(G)$-spectrum $\bA_P$ extends to a spectral Mackey functor for every finite group $G$.
\begin{proof}[Proof of \cref{thm:mackey-functor}]
 Due to \cref{prop:org-spectra} and \eqref{eq:transfers-compatible}, this is a direct consequence of \cite[Corollaries~4.7 \& 4.9]{coarsetrans}.
\end{proof}

\subsection{Strong additivity}

\begin{ddd}
 Let $(X_i,\cU_i,\cB_i)_{i \in I}$ be a family of $G$-bornological coarse spaces.
 The \emph{free union} $\bigsqcup_{i \in I}^\free X_i$ of this family is the following $G$-bornological coarse space:
 \begin{enumerate}
  \item The underlying $G$-set is the disjoint union $\bigsqcup_{i \in I} X_i$.
  \item The bornology is generated by those subsets $B$ for which $B \cap X_i \in \cB_i$ for all $i$
   and $B \cap X_i$ is non-empty for only finitely many $i$.
  \item The coarse structure is generated by the entourages $\bigsqcup_{i \in I} U_i$ {for all families of entourages $(U_i)_{i \in I}$ with $U_i \in \cU_i$.} \qedhere
 \end{enumerate}
\end{ddd}

Let $E$ be an equivariant coarse homology theory.
Let $(X_i)_{i \in I}$ be a family of $G$-bornological coarse spaces.
Since $(X_j,\bigsqcup_{j \neq i \in I} X_i)$ is a coarsely excisive pair,
we obtain by excision for every $j \in I$ a projection map $p_j \colon E(\bigsqcup_{i \in I} X_i) \to E(X_j)$.

\begin{ddd}\label{def:strongly-additive}
 Let $\cC$ be a cocomplete stable $\infty$-category which admits all products.
 An equivariant coarse homology theory $E \colon G\BC \to \cC$ is \emph{strongly additive} if for every family of $G$-bornological coarse spaces $(X_i)_{i \in I}$ the map
 \[ E(\bigsqcup_{i \in I}^\free X_i) \xrightarrow{(p_i)_i} \prod_{i \in I} E(X_i) \]
 is an equivalence.
\end{ddd}

\begin{rem}\label{rem:excision-transfers}
 Let $(X_i)_{i \in I}$ be a family of $G$-bornological coarse spaces.
 The inclusion $\inc_j \colon X_j \to \bigsqcup_{i \in I}^\free X_i$ of the $j$-th component is a bounded covering.
 If the equivariant coarse homology theory $E$ admits transfers,
 it follows formally that the projection map $p_j$ is equivalent to the transfer $\inc_j^*$ along $\inc_j$.
 See also \cite[Remark~2.62]{coarsetrans}.
\end{rem}

\begin{prop}\label{prop:a-homology-strongly-additive}
 The equivariant coarse homology theory $\AX^G_P$ is strongly additive.
\end{prop}
\begin{proof}
 Let $(X_i)_{i \in I}$ be a family of $G$-bornological coarse spaces.
 Using \cref{rem:excision-transfers}, the comparison map factors as
 \begin{align*}
  \wK(\ret{G}{\bigsqcup^\free_{i \in I} X_i}{P}{\lf})
  \xrightarrow{(\inc_i^*)_i} \wK(\prod_{i \in I} \ret{G}{X_i}{P}{\lf})
  \to \prod_{i \in I} \wK(\ret{G}{X_i}{P}{\lf}).
 \end{align*}
 In view of \cref{thm:products}, it suffices to observe that the transfer functors $\inc_i^*$ induce an exact equivalence
 \[ \ret{G}{\bigsqcup^\free_{i \in I} X_i}{P}{\lf} \xrightarrow{\sim} \prod_{i \in I} \ret{G}{X_i}{P}{\lf}. \]
 This follows directly from the definitions.
\end{proof}

\subsection{Split injectivity results}\label{rgioergegegerge}
We conclude this section by summarizing the results of the previous subsections and explaining how the axiomatic framework of \cite{injectivity} applies to prove split injectivity results for the $A$-theoretic assembly map.

\begin{ddd}
 A functor $M \colon \orb(G) \to \Sp$ is a \emph{CP-functor} if there exists an equivariant coarse homology theory $E$ such that the following holds:
 \begin{enumerate}
  \item $M$ is equivalent to $E(G_{can,min} \otimes (-)_{min,max})$;
  \item $E$ is continuous;
  \item $E$ is strongly additive;  
  \item $E$ extends to an equivariant coarse homology theory with transfers.
 \end{enumerate}
 We call $M$ a \emph{hereditary CP-functor} if $M \circ \res_\phi$ is a CP-functor for every surjective homomorphism $\phi \colon G \to Q$,
 where $\res_\phi \colon \orb(Q) \to \orb(G)$ denotes the functor restricting group actions along $\phi$.
\end{ddd}

\begin{theorem}\label{thm:a-cp}
 For every discrete group $G$ and principal $G$-bundle $P$, the functor $\bA_P$ is a hereditary CP-functor.
\end{theorem}
\begin{proof}
 $\bA_P$ is a CP-functor for every discrete group $G$ and every principal $G$-bundle $P$ by \cref{prop:org-spectra}, \cref{thm:a-homology}, \cref{prop:a-homology-strongly-additive} and \eqref{eq:transfers-compatible}.
 
 Let $\phi \colon G \to Q$ be a surjective homomorphism.
 Since there exists an isomorphism
 \[ P \times_G \res_\phi(S) \cong \ind_\phi(P) \times_Q S \]
 which is natural in $S$, we see that $\bA_P \circ \res_\phi \simeq \bA_{\ind_\phi(P)}$.
 Consequently, the previous paragraph implies that $\bA_P$ is a hereditary CP-functor.
\end{proof}

\begin{proof}[{Proof of \cref{thm:injectivity-linear,thm:split-injective-relative,thm:FDC}:}]
By \cref{thm:a-cp}, \cref{thm:injectivity-linear} follows from \cite[Corollary~2.11]{injectivity}, \cref{thm:split-injective-relative} is a consequence of \cite[Corollary~1.13]{injectivity}  {and \cref{thm:FDC} follows from \cite[Theorem~1.11]{injectivity}}.\end{proof}

\begin{proof}[Proof of \cref{thm:relhyp}]
	Let $\ccH$ be the smallest family of subgroups of $G$ that contains all virtually cyclic subgroups and all $H_i$. Repeating the proof of \cite[Theorem~4.4]{Bartels-relhyp} (invoking \cite[Theorem~6-14]{ELPUW18} instead of \cite[Theorem~4.3]{Bartels-relhyp}), the assembly map $\ass_{\bA_P}^\ccH$ is an equivalence. By the transitivity principle \cite[Proposition~11.2]{UW}, the relative assembly map $\ass_{\bA_P}^{\ccH\cap\FDC,\ccH}$ is an equivalence (here we have to use the assumptions on the groups $H_i$).
	By \cite[Theorem~1.11]{injectivity} and \cref{thm:a-cp}, the relative assembly map $\ass_{\bA_P}^{\Fin,\ccH\cap\FDC}$ admits a left inverse. The theorem now follows by combining these results.
\end{proof} 

\bibliographystyle{alpha}
\bibliography{coarseA}

\end{document}